\documentclass[a4paper,12pt,reqno]{amsart}

\usepackage{amsthm}
\usepackage{amsmath}
\usepackage{amssymb}
\usepackage{mathrsfs}
\usepackage{latexsym}
\usepackage{exscale}
\usepackage{geometry} 
\usepackage{graphicx} 
\usepackage[utf8]{inputenc}
\usepackage{verbatim}
\usepackage{enumerate} 
\usepackage{ esint } 
\usepackage{fancyhdr}

\usepackage{hyperref}

\usepackage{color}
\definecolor{red}{rgb}{1,0.1,0.1}
\definecolor{blue}{rgb}{0.1,0.1,1}
\definecolor{vb}{RGB}{160,32,240}

\newcommand{\vb}

\usepackage{xcolor}
\definecolor{violetafuerte}{HTML}{8000C0}

\theoremstyle{plain}

\newtheorem*{thm*}{Theorem}
\newtheorem*{prop*}{Proposition}
\newtheorem*{lemma*}{Lemma}

\numberwithin{equation}{section}
\newtheorem{thm}{Theorem}[section]
\newtheorem{lemma}[thm]{Lemma}
\newtheorem{corol}[thm]{Corollary}
\newtheorem{prop}[thm]{Proposition}

\theoremstyle{remark}
\newtheorem{remark}[thm]{Remark}

\theoremstyle{definition}

\newtheorem*{mydef*}{Definition}
\newtheorem{mydef}[thm]{Definition}
\newtheorem{ejem}[thm]{Example}

\calclayout

\allowdisplaybreaks

\newcommand{\R}{\mathbb{R}^n}

\begin{document}

\title[]{Weighted boundedness of maximal operator associated with invertible linear transformations on variable Lebesgue spaces}

\author[G.~Iba\~{n}ez~Firnkorn]{Gonzalo Iba\~{n}ez-Firnkorn}
\address{(G.~Iba\~{n}ez~Firnkorn) 
Instituto de Matemática (INMABB), Departamento de Matemática, Universidad Nacional del Sur (UNS)-CONICET
\\Bahía Blanca, Argentina}
\email{gonzalo.ibanez@uns.edu.ar}

\author[L.~A.~Vallejos]{Lucas Alejandro Vallejos}
\address{(L.~A.~Vallejos)  FaMAF \\ Universidad Nacional de C\'ordoba \\
CIEM (CONICET) \\ 5000 C\'ordoba, Argentina}
\email{lucas.vallejos@unc.edu.ar}

\begin{abstract}
In this paper, we study the boundedness of $M_{A^{-1}}$, the maximal operator involves an invertible matrix, on the weighted variable spaces $L^{p(\cdot)}(w)$, where the weights are in an appropriate class. In general, we obtain the necessary conditions for the weight classes and exponents. However, for certain types of invertible matrices, we characterize the weights such that the boundedness holds.
\end{abstract}

\thanks{The first author is partially supported by CONICET and SGCyT-UNS. The second author is partially supported by
CONICET and SECYT-UNC}



\maketitle

\section{Introduction}

Let $T$ be an invertible linear transformation and $A$ be the invertible matrix $n\times n$ associated to $T$.
 We focus on the study of the maximal function $M_{A^{-1}}$ defined by
\[M_{A^{-1}}f(x)=\sup_{AQ \ni x} \frac1{|Q|}\int_Q | f |.
\]
where the supremum is taken over all cubes $Q$ with sides parallel to the axes. 
Observe that $M_{A^{-1}}f(x)=Mf(A^{-1}x)$, where $M$ is the Hardy-Littlewood maximal operator. These operators appear in the study of operators with several singularities, for example,  
\[T_{\alpha,m}f( x)=\int_{\mathbb{R}} \frac{f(y)}{|x-A_1y|^{\beta_1}\cdots|x-A_my|^{\beta_m}}\, dy\]
where $A_i$ are invertible matrices such that $A_i-A_j$ is invertible for $i\not=j$ and $\beta_1+\cdots +\beta_m=n-\alpha$ with $0 \leq \alpha<n$. The operator $T_{\alpha,m}$ is controlled in norm by the sum of $M_{A_i^{-1},\alpha}$, the fractional version of the maximal defined above. See, e.g \cite{IFR18,IFRV22,RU13}.  This type of operators are widely studied in weighted Lebesgue spaces and unweighted variable Lebesgue spaces, see \cite{IFRV22,UV20b} and their references.

In \cite{IFRV22}, the first author, Riveros and Vidal defined the class of weights $\mathcal{A}_{A,p}$: a weight $w$ belongs to class $\mathcal{A}_{A,p}$ if 
$$[w]_{\mathcal{A}_{A,p}}=\sup_Q \left(\frac1{|Q|}\int_{Q} w_A \right) \left(\frac1{|Q|}\int_Q w^{-\frac1{p-1}}\right)^{p-1}<\infty$$
where the supremum is over all cubes $Q$ with side parallel to the axes.
Observe that if $A=I$ is the identity matrix, then the class $\mathcal{A}_{A,p}$ coincides with the classical Muckenpoupt class $\mathcal{A}_p$.

They characterized the weights for which the maximal operator $M_{A^{-1}}$ is bounded from $L^p(w)$ to $L^{p,\infty}(w)$. 
Recently in \cite{IF26}, the first author studied the boundedness for $M_{A^{-1}}$ on $L^{p}(w)$. More precisely, for the case  of $A$ is a invertible matrix such that $A^k=I$, it was proved 
$$\mathcal{A}_{A,p}=\{w \in \mathcal{A}_{p} : w(Ax) \lesssim w(x) \; a.e. \,x\in \mathbb{R}^n\}$$
and the maximal operator $M_{A^{-1}}$ is bounded on $L^p(w)$ if and only if $w\in \mathcal{A}_{A,p}$.

 In this paper, we study the boundedness of the maximal operator $M_{A^{-1}}$ in variable Lebesgue spaces with appropriate weights. 
The weighted variable Lebesgue spaces $L_{w}^{p(\cdot)}(\mathbb{R}^n)$ is defined by the norm $\left\Vert f\right\Vert _{p(\cdot ),w}=\left\Vert fw\right\Vert _{p(\cdot )}$, for precise definition see Section 2. 

In \cite{CUDH11, CUFN12}, the authors studied the boundedness of the Hardy-Littlewood maximal operator on $L_{w}^{p(\cdot)}(\mathbb{R}^n)$. 
They consider the class of weights $A_{p(\cdot)}$, defined as
$$[w]_{\mathcal{A}_{p(\cdot)}}=\sup_Q |Q|^{-1} \|w\chi_Q\|_{p(\cdot)}\|w^{-1}\chi_Q\|_{p'(\cdot)}<\infty,$$
where the supremum is over all cubes $Q$ with side parallel to the axes. They proved that if $p(\cdot)\in \mathcal{P}^{\log}$ such that $1<p_-\leq p_+ < \infty $, then  the Hardy-Littlewood maximal operator is bounded on $L^{p(\cdot)}_w$ if and only if $w\in A_{p(\cdot)}$.

For the study of the weighted boundedness of $M_{A^{-1}}$ in variable Lebesgue spaces, we define the class of weights $\mathcal{A}_{A,p(\cdot)}$. We said that a weight $w$ belongs to class $\mathcal{A}_{A,p(\cdot)}$
$$[w]_{\mathcal{A}_{A,p(\cdot)}}=\sup_Q |Q|^{-1} \|w_A\chi_Q\|_{p(\cdot)}\|w^{-1}\chi_Q\|_{p'(\cdot)}<\infty,$$
where the supremum is over all cubes $Q$ with the side parallel to the axes. Observe that if $A=I$ the class $\mathcal{A}_{A,p(\cdot)}$ coincides with $\mathcal{A}_{p(\cdot)}$ with above. If $p(\cdot)\equiv p$ is constant, then the class $\mathcal{A}_{A,p(\cdot)}$ coincides with the class $\mathcal{A}_{A,p}$.
 
The first main result is the necessity of the class of weights in the boundedness of the maximal operator for any matrix invertible $A$,
\begin{thm}\label{Necesidad-w}
  Let $w$ be a weight, $A$ be an invertible matrix, and $p(\cdot)\in \mathcal{P}^{\log}$ such that $p(Ax)=p(x)$ a.e. $x\in \mathbb{R}^n$. If $M_{A^{-1}}$ is bounded on $L^{p(\cdot)}(w)$, then $w\in \mathcal{A}_{A,p(\cdot)}$.
\end{thm}

The second result is the necessity of  relation between $p(Ax)$ and $p(x)$.
\begin{thm}\label{Necesidad-p}
Let $A$ be an invertible matrix and let $p(\cdot)\in \mathcal{P}(\mathbb{R}^n)$. Suppose that
there exists $x_0\in\mathbb{R}^n$ such that $p(\cdot)$ is continuous at $x_0$
and at $Ax_0$, and
\[
p(Ax_0)>p(x_0).
\]
Then the operator $M_{A^{-1}}$ is not bounded on $L^{p(\cdot)}(\mathbb{R}^n)$.
\end{thm}

\begin{remark}
The previous theorem states that $p(Ax)\leq p(x) \text{ for a.e. } x\in \mathbb{R}^n$ is a necessary condition for the maximal operator associated with the matrix $A^{-1}$ to be bounded on $L^{p(\cdot)}$.
Indeed, $p(Ax)< p(x)$ is a possible condition. In \cite{UV20b} the authors showed that if $p(\cdot)\in \mathcal{P}^{\log}$ such that $p(Ax)\leq p(x)$ a.e. $x\in \mathbb{R}^n$ then $\|g \circ A^{-1}\|_{p(\cdot)}\lesssim \|g \|_{p(\cdot)}$. 
Observe that $M_{A^{-1}}f=Mf \circ A^{-1}$. Using $g=Mf$ in the previous argument with $p(\cdot) \in \mathcal{P}^{\log}$ and $p(Ax)<p(x)$ a.e. $x\in \mathbb{R}^n$, we obtain the boundedness $M_{A^{-1}}$ on $L^{p(\cdot)}(\mathbb{R}^n)$. For example, we can take $A=2I$ and the exponent as in Section 6.
\end{remark}

For the boundedness of the maximal operator $M_{A^{-1}}$, as in the case of constant exponent, we  consider a particular type of matrix, for which the weights ensuring boundedness can be characterized.

\begin{thm}\label{Principal}
    Let $A$ be an invertible matrix such that $A^k=I$ for some $k\in \mathbb{N}$ and $p(\cdot) \in \mathcal{P}^{\log}$ such that $1<p_-\leq p_+<\infty$ and $p(Ax)=p(x)$ a.e. $x \in \mathbb{R}^n$. Then  $w\in \mathcal{A}_{A,p(\cdot)}$  if and only if $M_{A^{-1}}$ is bounded on $L^{p(\cdot)}_w$.
\end{thm}

Observe that this result characterizes the weights $w$ for which the maximal operator $M_{A^{-1}}$ is bounded on $L^{p(\cdot)}(w)$, extending the results of \cite{CUDH11,CUFN12,IF26}. An application directly related is  the boundedness of the operators $T_{\alpha,m}$ studied in \cite{IFRV22} and their references.

The remainder of this paper is organized as follows. In Section 2, we present some needed preliminaries results. In Section 3, we prove Theorem \ref{Necesidad-w} and Theorem \ref{Necesidad-p}, which establish necessary conditions related to our hypotheses. In Section 4, we prove Theorem \ref{Principal}, which establishes the boundedness of $M_{A^{-1}}$ on $L^{p(\cdot)}_w$, $w\in \mathcal{A}_{A,p(\cdot)}$.

In Section 5, we study several properties of the weights in $\mathcal{A}_{A,p(\cdot)}$ class, extending the properties provided in the constant context.

Finally, in Section 6, we present a method for constructing weights in \(\mathcal{A}_{p(\cdot)}\) and use it to obtain weights in \(\mathcal{A}_{A,p(\cdot)}\) and provide examples satisfying the hypotheses of Theorem \ref{Principal}.  

Throughout this paper, all notation is standard of will defined as needed. In order to emphasize that we are dealing with variable exponents, we will always write $p(\cdot)$ instead of $p$ to denote an exponent function. If we write $F\lesssim G$ then we mean that there exists a positive constant $c$ such that $F\leq c \, G$, the constant $c$ may depend of the matrix, the exponent and dimension  and if we write $A\approx B$ then we mean that there exist $c_1,c_2$ positive constants such that $c_1B\leq A\leq c_2B$.

\section{Preliminaries}

\subsection{Variable Lebesgue spaces}
Let $\mathcal{P}(\R)$ be  the family of  measurable functions $p(\cdot): \R \to [1,\infty)$.
Given $p(\cdot)\in \mathcal{P}(\R)$, for a measurable set $\Omega\subset\mathbb R^n$ of positive measure, we denote
\[
p_-(\Omega)=\operatorname*{ess\,inf}_{x\in\Omega}p(x),
\qquad
p_+(\Omega)=\operatorname*{ess\,sup}_{x\in\Omega}p(x).
\]
When $\Omega={\R}$, we only write $p_{-}$ and $p_{+}$ respectively.

For each $p(\cdot)\in \mathcal{P}(\R)$ we have the conjugate exponent $p'(\cdot)$ and the following modular associated

\begin{equation*}
\mathit{\rho }_{p(\cdot )}\left( f \right) =\int_{\R}  |f(x)|^{p(x)}\,dx.
\end{equation*}

Given $p(\cdot)\in \mathcal{P}(\R)$, let $L^{p(\cdot)}(\R)$ be the Banach space of measurable functions $f$ on $\R$ such that, for some  $\lambda >0$,  the modular $\rho_{p(\cdot)}(f/\lambda)<\infty$ is finite, with norm
\begin{equation*}
\left\Vert f\right\Vert _{p(\cdot )}=\inf \{\lambda >0:~%
\mathit{\rho }_{p(\cdot )}\left( f/\lambda \right) \leq 1\}\text{.}
\end{equation*}

These spaces are known as {\it variable Lebesgue spaces} and generalize the classical Lebesgue spaces $L^p(\R)$.
They have been widely studied lately, see for example \cite{CCUF07,CUFN03,DHHRlibro}.

\begin{mydef}
Given $\Omega \subseteq \mathbb{R}^n$ and a function  $r(\cdot ):\Omega \rightarrow \mathbb{R}$,
we say that $r(\cdot )~$is locally $\log -H\ddot{o}lder$ continuous, we denote by $r(\cdot) \in LH_0(\Omega)$, if there
exists a constant $C_{0}$ such that for all $x,y\in \Omega $, $|x-y|~<\frac{1}{2}$, 
\begin{equation}\label{LH_0}
    |r(x)-r(y)|~\leq \frac{C_{0}}{-\log (|x-y|)}.
\end{equation}

We say that $r(\cdot )$ is $\log -H\ddot{o}lder$ continuous at infinity, we denote by $r(\cdot) \in LH_{\infty}(\Omega)$, if
there exist constants $C_{\infty }$ and $r_{\infty }$ such that 
\begin{equation}\label{LH_infty}
    |r(x)-r_{\infty }|~\leq ~\frac{C_{\infty }}{\log (e+|x|)},
\end{equation}
for all $x\in \Omega $.

We denote by $\mathcal{P}^{\log}$ the family of exponent function $r(\cdot)\in LH_0(\mathbb{R}^n) \cap LH_{\infty}(\mathbb{R}^n)$.
\end{mydef}

In \cite{CUFlibro} the authors prove the following lemmas.

\begin{lemma} \label{Lema-LH_0}
Let $r(\cdot)\in LH_0(\mathbb{R}^n)$. Then there exists a constant
$C>0$, depending of $n$, such that for every cube $Q\subset\mathbb{R}^n$ and almost every
$x\in Q$,
\[
    |Q|^{r(x)-r_+(Q)}\leq C
    \qquad\text{and}\qquad
    |Q|^{r_-(Q)-r(x)}\leq C.
\]
\end{lemma}

\begin{lemma} \label{Lema-LH_infty}
Given a set $\Omega \subseteq \mathbb{R}^n$, two exponents $r(\cdot)$ and $s(\cdot)$ such that 
\[|s(y)-r(y)| \leq  \frac{C}{\log(e+|y|)}\]
and a non-negative measure $\mu$. Then for every $t\geq 1$ there
exists a constant $C>0$ such that for all  functions $F$ with
$0\leq F(y)\leq1$,
\[
    \int_{\Omega} F(y)^{s(y)}\,d\mu(y)
    \leq
    C\left(
        \int_{\Omega} F(y)^{r(y)}\,d\mu(y)
        +
        \int_{\Omega} (e+|y|)^{-Nr_-}\,d\mu(y)
    \right).
\]
\end{lemma}

It is well known that the Hardy-Littlewood maximal operator is bounded on $L^{p(\cdot)}(\mathbb{R}^n)$ when $r(\cdot)\in \mathcal{P}^{\log}$.

We recall several basic properties of variable Lebesgue spaces. We next state another useful property concerning the norm-modular comparison.

\begin{lemma}
\label{norma-modular}
Let $p(\cdot)\in \mathcal{P}(\R)$ with $p_+<\infty$, and let
$\Omega\subset\mathbb R^n$ be measurable set with
$|\Omega|>0$.
Then the following assertions hold.
\begin{enumerate}
\item If $\|f\chi_\Omega\|_{p(\cdot)}\leq1$, then
\[
\|f\chi_\Omega\|_{p(\cdot)}^{p_+(\Omega)}
\leq
\int_\Omega |f(x)|^{p(x)}\,dx
\leq
\|f\chi_\Omega\|_{p(\cdot)}^{p_-(\Omega)}.
\]

\item If $\|f\chi_\Omega\|_{p(\cdot)}\geq1$, then
\[
\|f\chi_\Omega\|_{p(\cdot)}^{p_-(\Omega)}
\leq
\int_\Omega |f(x)|^{p(x)}\,dx
\leq
\|f\chi_\Omega\|_{p(\cdot)}^{p_+(\Omega)}.
\]
\end{enumerate}
In particular, if $\|f\|_{p(\cdot)}\leq1$, then
\[
\int_{\mathbb R^n}|f(x)|^{p(x)}\,dx
\leq \|f\|_{p(\cdot)}.
\]
\end{lemma}

\begin{lemma}
\label{norma-1-modular}
Let $p(\cdot)\in \mathcal{P}(\R)$ with $p_+<\infty$.
If $\|f\|_{p(\cdot)}=1$, then
\[
\int_{\mathbb R^n}|f(x)|^{p(x)}\,dx=1.
\]
\end{lemma}

As a consequence of Lemma~\ref{norma-1-modular},
for every nonzero $f\in L^{p(\cdot)}(\mathbb R^n)$,
\[
\rho_{p(\cdot)}
\left(\frac{f}{\|f\|_{p(\cdot)}}\right)=1.
\]

The variable version of Hölder's inequality is the following

\begin{lemma}[H\"older's inequality]
\label{Holder}
Let $p(\cdot) \in \mathcal{P}(\R)$
and let $p'(\cdot)$ be the conjugate exponent of $p(\cdot)$.
There exists a constant $C>0$ such that
\[
\int_{\mathbb R^n}|f(x)g(x)|\,dx
\leq
C\|f\|_{p(\cdot)}\|g\|_{p'(\cdot)}
\]
for every $f\in L^{p(\cdot)}(\mathbb R^n)$ and
$g\in L^{p'(\cdot)}(\mathbb R^n)$.
\end{lemma}

The following lemma describes a property related to the homogeneity of the norm in the exponent.

\begin{lemma}\label{dilationExpo} 
Let $p(\cdot )\in \mathcal{P}(\mathbb{R}^n)$. Then for all $s$ with $1/p^{-}\leq s<\infty $, 
\begin{equation*}
\Vert |f|^{s}\Vert _{p(\cdot)}=\Vert f\Vert_{sp(\cdot)}^{s}.
\end{equation*}
\end{lemma}

In \cite{UV20b} the authors prove the following result.

\begin{prop}\label{UV}\cite{UV20b}
Let $A$ be an invertible matrix and let  $p(\cdot )\in \mathcal{P}(\mathbb{R}^{n})$. If $p(Ax)=p(x)$ a.e. $x$ then 
\[
\left\Vert f\circ A\right\Vert _{p(\cdot )}\leq \max\{1,|\det(A)|^{\frac1{p_-}}\} \left\Vert f\right\Vert
_{p(\cdot )}{,}
\]
for all $f\in L^{p(\cdot)}(\mathbb{R}^{n})$. 
\end{prop}



\subsection{Weighted Variable Lebesgue spaces}

A locally integrable function $w:\mathbb{R}^n \rightarrow [0, \infty)$ is called a weight. For $p(\cdot) \in \mathcal{P}(\mathbb{R}^n)$, we define the weighted modular

\begin{equation}
    \rho _{p(\cdot ),w}(f)=\int_{\mathbb{R}^{n}}|f(x)|^{p(x)}w(x)dx.
\end{equation}

The space $L_{w}^{p(\cdot )}(\mathbb{R}^{n})$ consists of all measurable
functions $f$ such that $\rho_{p(\cdot ),w}(f/\lambda )<\infty $ for some $%
\lambda >0$. 

The norm $\left\Vert \cdot \right\Vert _{p(\cdot ),w}$ of $L_{w}^{p(\cdot )}(%
\mathbb{R}^{n})$ satisfies 
\begin{equation}
    \left\Vert f\right\Vert _{p(\cdot ),w}=\left\Vert fw\right\Vert _{p(\cdot )}
\end{equation}

Note that since $w\in L^{p(\cdot)}_{\text{loc}}$, all characteristic functions of cubes belong to $L^{p(\cdot)}_{w}$. This implies that the restriction $L^{p(\cdot)}_w(Q)$ is a Banach function space for every cube $Q$. Moreover, the associate space of $L_{w}^{p(\cdot)}(Q)$ is given by $L_{w^{-1}}^{p'(\cdot)}(Q)$, 
see \cite[Theorem 2.7.4]{DHHRlibro}. The following norm conjugate formula holds
for $L_{w}^{p(\cdot)}(Q)$, see \cite[Corollary 2.7.5]{DHHRlibro},
\[
\frac{1}{2}\|f\|_{L_{w}^{p(\cdot)}(Q)}
\leq
\sup_{\substack{
\|g\|_{L_{w^{-1}}^{p'(\cdot)}(Q)}\leq 1
}}
\int_Q |f(x)|\,|g(x)|\,dx
\leq
2\|f\|_{L_{w}^{p(\cdot)}(Q)}.
\]

Therefore, by monotone convergence the norm conjugate formula also holds
for $L_{w}^{p(\cdot)}(\mathbb{R}^n)$. By the same argument $L_{w^{-1}}^{p'(\cdot)}(\mathbb{R}^n)$ is the
associate space of $L_{w}^{p(\cdot)}(\mathbb{R}^n)$.

In \cite{CUP25}, Cruz-Uribe and Penrod proved the a variable reverse Hölder, in particular showed that the class $\mathcal{A}_{p(\cdot)}$ is left open.

\begin{prop}\cite{CUP25}\label{Ap-OpenLeft}
Given $p(\cdot)\in \mathcal{P}^{\log}$ with $p_+<\infty$ and $w\in \mathcal{A}_{p(\cdot)}$, then 
\begin{enumerate}
\item[i)] there exists $r>1$ such that for all $s\in [1,r]$, $w\in A_{sp(\cdot)}$.
\item[ii)] there exists $r>1$ such that for all $s\in [1,r]$, if $q'(\cdot)=sp'(\cdot)$ then $w\in A_{q(\cdot)}$.
\end{enumerate}
\end{prop}

\subsection{Maximal operator $M_{A^{-1}}$ and weights}

In the context of classical Lebesgue spaces, this type of maximal operator was widely studied, we refer to \cite{IF26,IFR18,IFRV22}.  

In \cite{IFRV22}, the authors introduce the class of weight $\mathcal{A}_{A,p}$, a weight $w$ belong to $\mathcal{A}_{A,p}$ if 
\[ [w]_{\mathcal{A}_{A,p}}= \sup_Q \left(\frac1{|Q|}\int_Q w(Ax) \, dx \right)\left(\frac1{|Q|}\int_Q w(x)^{-\frac1{p-1}} \, dx \right)^{p-1}<\infty\]
where the supremum is over all cubes $Q$ with side parallel to the axes.

They showed that $M_{A^{-1}}$ is bounded from $L^p(w)$ into $L^{p,\infty}(w)$ if and only if $w\in \mathcal{A}_{A,p}$.

Now, we present several properties of this class of weights

\begin{prop}
    \cite{IFR18}
    Let $A$ be a invertible matrix and $w$ be a non-negative function. Then, $M_{A^{-1}}w(x)=M(w_A)(x)$.
\end{prop}

\begin{prop}\cite{IFRV22}\label{Propiedades}
Let $A, B$ be invertible matrices. 
\begin{enumerate}[i)]
\item If $w\in\mathcal{A}_{A,p}$ then $w(Ax)\leq [w]_{\mathcal{A}_{A,p}}w(x)$ a.e. $x\in \mathbb{R}^n$.
\item If $w\in\mathcal{A}_{A,p}$ and $w\in\mathcal{A}_{B,p}$ then $w\in\mathcal{A}_{AB,p}$.
\end{enumerate}
\end{prop}

Recently, in \cite{IF26} the first author proved that for some particular matrix $A$ the class $\mathcal{A}_{A,p}$ is a subclass of the classical Muckenhoupt class. More precisely,

\begin{prop}\cite{IF26} \label{clase A nilpotent}
    Let $A$ be an invertible matrix such that $A^k=I$ for some $k\in \mathbb{N}$. Then $\mathcal{A}_{A,p}=\left\{w\in \mathcal{A}_p : w(Ax)\lesssim w(x) \;a.e. \, x\in \mathbb{R}^n\right\}$.
\end{prop}

Moreover, in \cite{IF26} the following characterization for the weights was proved
\begin{thm}\cite{IF26} Let $A$ be an invertible matrix such that $A^k=I$ for some $k\in \mathbb{N}$. Then $M_{A^{-1}}$ is bounded on $L^p(w)$ if and only if $w\in \mathcal{A}_{A,p}$.
\end{thm}

\section{Necessary condition}
In this section, we prove the necessary conditions Theorem \ref{Necesidad-w} and Theorem \ref{Necesidad-p}.

\begin{proof}[Proof of Theorem \ref{Necesidad-w}]
Since $M_{A^{-1}}$ is bounded on $L^{p(\cdot)}(w)$ we have 
for each $f\in L^{p(\cdot)}(w)$, 
\[\|M_{A^{-1}}f\|_{L^{p(\cdot)}(w)}\lesssim \|f\|_{L^{p(\cdot)}(w)}\]

Since $p(Ax)=p(x)$ a.e. $x$, by Proposition \ref{UV} we have
\[\|M_{A^{-1}}f\|_{L^{p(\cdot)}(w)}=
\|wM_{A^{-1}}f\|_{L^{p(\cdot)}}
\gtrsim \| \left[w M_{A^{-1}}f\right] \circ A\|_{L^{p(\cdot)}}
=\|w_A Mf\|_{L^{p(\cdot)}}
\]

For any cube $Q\subset \mathbb{R}^n$ we have $Mf\geq  \left(\frac1{|Q|} \int_Q |f|\right) \chi_Q$ and obtain 
\[\|w_A Mf\|_{L^{p(\cdot)}}\geq \left( \frac1{|Q|} \int_Q |f| \right)\|\chi_Q w_A\|_{p(\cdot)}\]

Now, taking 
\[f(x)=\lambda \left(\frac{w^{-1}(x)}{\lambda}\right)^{p'(x)} \chi_Q(x)
\]
with $\lambda$ to determine later. 
Observe that
\[\frac1{|Q|} \int_Q |f|
=\frac{\lambda}{|Q|}  \int_Q \left(\frac{w^{-1}(x)}{\lambda}\right)^{p'(x)}\,dx
=\frac{\lambda}{|Q|} \rho_{p'(\cdot)} \left(\frac{w^{-1}\chi_Q}{\lambda}\right)
\]
and 
\begin{align*}
\rho_{p(\cdot)}(fw)
&= \int_Q \left(\frac{w^{-p'(x)}(x)w(x)}{\lambda^{p'(x)-1}}\right)^{p(x)} \, dx
=\int_Q \left(\frac{w^{-1}(x)}{\lambda}\right)^{p'(x)} \,dx
\\&=\rho_{p'(\cdot)} \left(\frac{w^{-1}\chi_Q}{\lambda}\right)
\end{align*}


Now, taking $\lambda = \|w^{-1}\chi_Q\|_{p'(\cdot)}$. 
Then, $\rho_{p(\cdot)}(fw) = \rho_{p'(\cdot)} \left(\frac{w^{-1}\chi_Q}{\|w^{-1}\chi_Q\|_{p'(\cdot)}}\right)=1$ 
this implies $\|fw\|_{p(\cdot)} \leq 1$.

On the other hand, 
\begin{align*}
1\gtrsim \|M_{A^{-1}}f\|_{L^p(\cdot)(w)} 
&\geq \left( \frac1{|Q|} \int_Q |f| \right)\|\chi_Q w_A\|_{p(\cdot)}
\\&= \frac{\lambda}{|Q|} \rho_{p'(\cdot)} \left(\frac{w^{-1}\chi_Q}{\lambda}\right)\|\chi_Q w_A\|_{p(\cdot)}
\\& = \frac{\|w^{-1}\chi_Q\|_{p'(\cdot)}}{|Q|} \rho_{p'(\cdot)} \left(\frac{w^{-1}\chi_Q}{\|w^{-1}\chi_Q\|_{p'(\cdot)}}\right)\|\chi_Q w_A\|_{p(\cdot)}
\\& = \frac{\|w^{-1}\chi_Q\|_{p'(\cdot)}}{|Q|} \|\chi_Q w_A\|_{p(\cdot)}
\end{align*}

Now, taking account the estimates above we obtain
\begin{align*}
\frac{\|w^{-1}\chi_Q\|_{p'(\cdot)}}{|Q|} \|\chi_Q w_A\|_{p(\cdot)}
&\leq \|M_{A^{-1}}f\|_{L^p(\cdot)(w)} 
\\&\leq \|M_{A^{-1}}\|_{L^{p(\cdot)}(w)\to L^{p(\cdot)}(w)} \|fw\|_{p(\cdot)} \\& \leq \|M_{A^{-1}}\|_{L^{p(\cdot)}(w)\to L^{p(\cdot)}(w)}
\end{align*}
for any cube $Q$. Then $[w]_{\mathcal{A}_{A,p(\cdot)}}\leq \|M_{A^{-1}}\|_{L^{p(\cdot)}(w)\to L^{p(\cdot)}(w)}$ and $w \in \mathcal{A}_{A,p(\cdot)}$.

\end{proof}

\begin{proof}[Proof of Theorem \ref{Necesidad-p}]

Set
\begin{equation*}
p_0=p(x_0),\qquad p_1=p(Ax_0),
\end{equation*}
so that $p_0<p_1$. Fix $\varepsilon>0$ sufficiently small such that
\[
p_0+\varepsilon<p_1-\varepsilon.
\]

By the continuity of $p$ at $x_0$, there exists $r_0>0$ such that,
for $0<r<r_0$ and $x\in Q_r:=Q(x_0,r)$,
\begin{equation}\label{1}
p_0-\varepsilon\leq p(x)\leq p_0+\varepsilon.
\end{equation}
Similarly, by the continuity of $p$ at $Ax_0$, after decreasing $r_0$
if necessary, we have
\begin{equation}\label{2}
p_1-\varepsilon\leq p(x)\leq p_1+\varepsilon,
\qquad x\in AQ_r.
\end{equation}
Indeed, $AQ_r$ shrinks to $Ax_0$ as $r\to0$.

For $r$ sufficiently small, we may assume that
\[
|Q_r|<1
\qquad\text{and}\qquad
|AQ_r|<1.
\]

First, observe that
\[
\chi_{Q_r}\in L^{p(\cdot)}(\mathbb{R}^n),
\]
since
\[
\rho_{p(\cdot)}(\chi_{Q_r})
=
\int_{\mathbb{R}^n}|\chi_{Q_r}(x)|^{p(x)}\,dx
=
|Q_r|<\infty.
\]

We now estimate $\|\chi_{Q_r}\|_{p(\cdot)}$. Let
\[
\lambda=|Q_r|^{1/(p_0+\varepsilon)}.
\]
Since $|Q_r|<1$, we have $0<\lambda<1$. By \eqref{1},
\[
\lambda^{-p(x)}
\leq
\lambda^{-(p_0+\varepsilon)}
=
|Q_r|^{-1},
\qquad x\in Q_r.
\]
Hence
\[
\int_{Q_r}\lambda^{-p(x)}\,dx
\leq
|Q_r|\,|Q_r|^{-1}
=1,
\]
and therefore
\begin{equation}\label{3}
\|\chi_{Q_r}\|_{p(\cdot)}
\leq
|Q_r|^{1/(p_0+\varepsilon)}.
\end{equation}

On the other hand, take
\[
\lambda=|Q_r|^{1/(p_0-\varepsilon)}.
\]
Again, $0<\lambda<1$, and by \eqref{1},
\[
\lambda^{-p(x)}
\geq
\lambda^{-(p_0-\varepsilon)}
=
|Q_r|^{-1},
\qquad x\in Q_r.
\]
Thus
\[
\int_{Q_r}\lambda^{-p(x)}\,dx\geq1,
\]
which implies
\begin{equation}\label{4}
\|\chi_{Q_r}\|_{p(\cdot)}
\geq
|Q_r|^{1/(p_0-\varepsilon)}.
\end{equation}
Consequently,
\begin{equation}\label{5}
|Q_r|^{1/(p_0-\varepsilon)}
\leq
\|\chi_{Q_r}\|_{p(\cdot)}
\leq
|Q_r|^{1/(p_0+\varepsilon)}.
\end{equation}

Similarly, using \eqref{2}, we obtain
\begin{equation}\label{6}
|AQ_r|^{1/(p_1-\varepsilon)}
\leq
\|\chi_{AQ_r}\|_{p(\cdot)}
\leq
|AQ_r|^{1/(p_1+\varepsilon)}.
\end{equation}

Now, if $x\in AQ_r$, then by the definition of $M_{A^{-1}}$,
\[
M_{A^{-1}}\chi_{Q_r}(x)
\geq
\frac{1}{|Q_r|}
\int_{Q_r}\chi_{Q_r}(y)\,dy
=1.
\]
Therefore,
\[
M_{A^{-1}}\chi_{Q_r}
\geq
\chi_{AQ_r},
\]
and hence
\begin{equation}\label{7}
\|M_{A^{-1}}\chi_{Q_r}\|_{p(\cdot)}
\geq
\|\chi_{AQ_r}\|_{p(\cdot)}.
\end{equation}

Combining \eqref{5}, \eqref{6}, and \eqref{7}, we obtain
\[
\frac{
\|M_{A^{-1}}\chi_{Q_r}\|_{p(\cdot)}
}{
\|\chi_{Q_r}\|_{p(\cdot)}
}
\geq
\frac{
|AQ_r|^{1/(p_1-\varepsilon)}
}{
|Q_r|^{1/(p_0+\varepsilon)}
}.
\]
Since $A$ is invertible,
\[
|AQ_r|=|\det A|\,|Q_r|,
\]
and there
\begin{equation}\label{8}
\frac{
\|M_{A^{-1}}\chi_{Q_r}\|_{p(\cdot)}
}{
\|\chi_{Q_r}\|_{p(\cdot)}
}
\geq
|\det A|^{1/(p_1-\varepsilon)}
|Q_r|^{
\frac{1}{p_1-\varepsilon}
-
\frac{1}{p_0+\varepsilon}
}.
\end{equation}

By our choice of $\varepsilon$,
\[
p_0+\varepsilon<p_1-\varepsilon,
\]
and hence
\[
\frac{1}{p_1-\varepsilon}
-
\frac{1}{p_0+\varepsilon}
<0.
\]
Since $|Q_r|\to0$ as $r\to0$, it follows from \eqref{8} that
\begin{equation}\label{9}
\frac{
\|M_{A^{-1}}\chi_{Q_r}\|_{p(\cdot)}
}{
\|\chi_{Q_r}\|_{p(\cdot)}
}
\longrightarrow+\infty
\qquad\text{as }r\to0.
\end{equation}

Suppose, by contradiction, that $M_{A^{-1}}$ were bounded on
$L^{p(\cdot)}(\mathbb{R}^n)$. Then there would exist a constant
$C>0$, independent of $f$, such that
\[
\|M_{A^{-1}}f\|_{p(\cdot)}
\leq C\|f\|_{p(\cdot)}
\]
for every $f\in L^{p(\cdot)}(\mathbb{R}^n)$. In particular, taking
$f=\chi_{Q_r}$, we would have
\[
\frac{
\|M_{A^{-1}}\chi_{Q_r}\|_{p(\cdot)}
}{
\|\chi_{Q_r}\|_{p(\cdot)}
}
\leq C
\]
for every sufficiently small $r$, contradicting \eqref{9}.

Therefore,
\[
M_{A^{-1}}
\text{ is not bounded on }
L^{p(\cdot)}(\mathbb{R}^n).
\]
\end{proof}

\section{Sufficient condition}

For the proof of the sufficient condition we need to define the following operators. Let $p(\cdot )\in \mathcal{P}(\mathbb{R}^{n})$. For every cube $Q\subset \mathbb{R}^{n}$ we define the average 
\begin{equation} \label{average}
f_{p(\cdot ),Q}
:=\frac{\left\Vert \chi _{Q}f\right\Vert
_{p(\cdot )}}{\left\Vert \chi _{Q}\right\Vert _{p(\cdot )}}
\end{equation}
and the maximal operator
\begin{equation} \label{averageMax}
M_{p(\cdot)}(f)(x)=\sup_{Q\ni x} f_{p(\cdot ),Q}
\end{equation}
\subsection{Auxiliary results}

In \cite{CUDH11} D. Cruz Uribe et al. studied this average and proved the following result

\begin{lemma}\label{LemaCUDH}\cite{CUDH11}
Let $p(\cdot)\in \mathcal{P}(\mathbb{R}^n)$ with $p_+<\infty$, $1\in A_{p(\cdot )}$ and let $w$ be a weight. Then $w\in A_{p(\cdot )}$
if and only if
\[1\lesssim w_{p(\cdot),Q}(w^{-1})_{p'(\cdot),Q} \lesssim [w]_{\mathcal{A}_{p(\cdot)}}.\]

for all cubes $Q$ in $\mathbb{R}^n$.
\end{lemma}

Diening et al. proved the following result for the maximal operator defined in \eqref{averageMax}.

\begin{thm}\cite{DHHRlibro}\label{Mq bounded}
Let $p(\cdot),q(\cdot),s(\cdot) \in \mathcal{P}^{\log}(\mathbb{R}^n)$ such that $p_+<\infty$,  $p(\cdot)=q(\cdot)s(\cdot)$ and $s_->1$. Then $M_{q(\cdot)}$ is bounded on $L^{p(\cdot)}(\mathbb{R}^n)$.
\end{thm}

We now prove the following result for our class of weights $\mathcal{A}_{A,p(\cdot )}$.

\begin{prop}\label{CLAIM 1}
Let $p(\cdot)\in \mathcal{P}(\mathbb{R}^n)$ with $p_+<\infty$, $A$ and $B$ be invertible matrices such that $p(Bx)=p(x)$. If $w\in \mathcal{A}_{A,p(\cdot)} \cap \mathcal{A}_{B,p(\cdot)}$ then $w\in \mathcal{A}_{AB,p(\cdot)}$.

Moreover, $[w]_{\mathcal{A}_{AB,p(\cdot )}}\leq [w]_{\mathcal{A}_{A,p(\cdot )}}[w]_{\mathcal{A}_{B,p(\cdot
)}}$.
\end{prop}

\begin{proof}
Let $Q$ be a cube. Since $w_{AB}(x)=w(ABx)=w_{A}(Bx)$, by Hölder
inequality,
\begin{align*}
&|Q|^{-1} \|w_{AB}\chi_Q\|_{p(\cdot)} 
\|w^{-1}\chi_Q\|_{p'(\cdot)}
\\&\leq|Q|^{-1} \|w_{AB}\chi_Q\|_{p(\cdot)} 
\frac{K_{p(\cdot)} \|w^{-1}_B \chi_Q\|_{p'(\cdot)}\|w_B\chi_Q\|_{p(\cdot)}}{|Q|}
\|w^{-1}\chi_Q\|_{p'(\cdot)}.
\end{align*}
Observe that
\begin{align*}
\|w_{AB}\chi_Q\|_{p(\cdot)} 
=\|(w \circ AB )  \chi_Q\|_{p(\cdot)}
=\|(w_A   \chi_{B^{-1}Q} ) \circ B \|_{p(\cdot)}
\end{align*}
and
\begin{align*}
\|w_{B}^{-1}\chi_Q\|_{p'(\cdot)} 
=\|(w^{-1} \circ B )  \chi_Q\|_{p'(\cdot)}
=\|(w^{-1}   \chi_{B^{-1}Q} ) \circ B \|_{p'(\cdot)}.
\end{align*}

We take a cube $R$ such that $B^{-1}Q\subset R$ and $|Q|\approx |R|$.
So, if  $p(Bx)=p(x)$ by the Proposition \ref{UV} and the definition of the class of weights $\mathcal{A}_{A,p(\cdot)}$ and $\mathcal{A}_{B,p(\cdot)}$ then we obtain
\begin{eqnarray*}
&& |Q|^{-1}\left\Vert w_{AB}\chi _{Q}\right\Vert _{p(\cdot )}\left\Vert
w^{-1}\chi _{Q}\right\Vert _{p^{\prime }(\cdot )} \\
&\lesssim &|Q|^{-1}\left\Vert (w_{A}\chi _{B^{-1}Q})\circ B\right\Vert
_{p(\cdot )}\frac{\left\Vert (w^{-1}\chi _{B^{-1}Q})\circ B\right\Vert
_{p^{\prime }(\cdot )}\left\Vert w_{B}\chi _{Q}\right\Vert _{p(\cdot )}}{|Q|}%
\left\Vert w^{-1}\chi _{Q}\right\Vert _{p^{\prime }(\cdot )} \\
&\lesssim &|Q|^{-1}\left\Vert w_{A}\chi _{B^{-1}Q}\right\Vert _{p(\cdot )}%
\frac{\left\Vert w^{-1}\chi _{B^{-1}Q}\right\Vert _{p^{\prime }(\cdot
)}\left\Vert w_{B}\chi _{Q}\right\Vert _{p(\cdot )}}{|Q|}\left\Vert
w^{-1}\chi _{Q}\right\Vert _{p^{\prime }(\cdot )} \\
&\lesssim &\left[ |Q|^{-1}\left\Vert w_{A}\chi _{B^{-1}Q}\right\Vert
_{p(\cdot )}\left\Vert w^{-1}\chi _{B^{-1}Q}\right\Vert _{p^{\prime }(\cdot
)}\right] ~\left[ |Q|^{-1}\left\Vert w_{B}\chi _{Q}\right\Vert _{p(\cdot
)}\left\Vert w^{-1}\chi _{Q}\right\Vert _{p^{\prime }(\cdot )}\right]  \\
&\lesssim &\left[ |R|^{-1}\left\Vert w_{A}\chi _{R}\right\Vert _{p(\cdot
)}\left\Vert w^{-1}\chi _{R}\right\Vert _{p^{\prime }(\cdot )}\right] \left[
|Q|^{-1}\left\Vert w_{B}\chi _{Q}\right\Vert _{p(\cdot )}\left\Vert
w^{-1}\chi _{Q}\right\Vert _{p^{\prime }(\cdot )}\right]  \\
&\lesssim &[w]_{\mathcal{A}_{A,p(\cdot )}}[w]_{\mathcal{A}_{B,p(\cdot
)}}<\infty .
\end{eqnarray*}
This imply that $[w]_{\mathcal{A}_{AB,p(\cdot )}}\leq [w]_{\mathcal{A}_{A,p(\cdot )}}[w]_{\mathcal{A}_{B,p(\cdot
)}}$ and $w\in \mathcal{A}_{AB,p(\cdot)}$.
\end{proof}

A consequence of the last proposition, we obtain the following result.

\begin{corol}\label{pesos-A-nilp} 
Let $p(\cdot)\in \mathcal{P}(\mathbb{R}^n)$ such that $p_+<\infty$ and $p(Ax)=p(x)$. If $A$ such that $A^k=I$ for some $k\in \mathbb{N}$ then 
$\mathcal{A}_{A,p(\cdot)}\subset \mathcal{A}_{p(\cdot)}$. 
Moreover, $[w]_{\mathcal{A}_{p(\cdot)}}\leq [w]_{\mathcal{A}_{A,p(\cdot)}}^k$
\end{corol}

The following lemma guaranties conditions for that $w_A\in \mathcal{A}_{p(\cdot)}$.

\begin{lemma}\label{w_A in Ap}
   Let $A$ be an invertible matrix with $|\det(A)|=1$, and let $p(\cdot) \in \mathcal{P}(\mathbb{R}^n)$ be such that $p_+<\infty$ and $p(Ax)=p(x)$ a.e. $x\in \mathbb{R}^n$. If $w \in \mathcal{A}_{p(\cdot)}$, then $w_A \in \mathcal{A}_{p(\cdot)}$ and $[w_A]_{\mathcal{A}_{p(\cdot)}}\leq [w]_{\mathcal{A}_{p(\cdot)}}\|A\|^n$.
\end{lemma}

\begin{proof}
Since $|\det (A)|=1$, $w\in \mathcal{A}_{p(\cdot )}$ and $p(Ax)=p(x)$ a.e. $x\in \mathbb{R}^n$, we analyze the quantity $\left\Vert w_{A}\chi _{B}\right\Vert _{p(\cdot )}$. Let $B=B(x_{0},r)$ and define $AB=\{Ay:y\in B\}$. Then
\begin{align*}
    \rho _{p(\cdot )}\left( \frac{w_{A}\chi _{B}}{\lambda }\right)
&=\int_{B}\left( \frac{w_{A}(x)}{\lambda }\right) ^{p(x)}dx
=\int_{B}\left( \frac{w(Ax)}{\lambda }\right) ^{p(x)}dx\\
&=\int_{AB}\left( \frac{w(y)}{\lambda }\right)^{p(A^{-1}y)}dy
=\int_{AB}\left( \frac{w(y)}{\lambda }\right)^{p(y)}dy.
\end{align*}

Since $AB\subset B(Ax_{0},\left\| A\right\| r)=B_{A}$, where $\|A\|$ is the operator norm, it follows that
\[
\rho _{p(\cdot )}\left( \frac{w_{A}\chi _{B}}{\lambda }\right) \leq
\int_{B_{A}}\left( \frac{w(y)}{\lambda }\right) ^{p(y)}dy=\rho_{p(\cdot
)}\left( \frac{w\chi _{B_{A}}}{\lambda }\right).
\]

So, taking $\lambda=\|w\chi_{B_A}\|_{p(\cdot)}$ we have $\rho_{p(\cdot
)}\left( \frac{w\chi _{B_{A}}}{\lambda }\right)\leq 1$
then $\rho _{p(\cdot )}\left( \frac{w_{A}\chi _{B}}{\lambda }\right) \leq 1$. By definition of the norm in $L^{p(\cdot )}(\mathbb{R}^n)$, $\left\Vert w_{A}\chi _{B}\right\Vert _{p(\cdot )}\leq\left\Vert w\chi _{B_{A}}\right\Vert _{p(\cdot )}$. 

Analogously, since $p^{\prime }(Ax)=p^{\prime }(x)$ a.e. $x\in \mathbb{R}^n$, we also obtain        $\left\Vert w_{A}^{-1}\chi _{B}\right\Vert _{p^{\prime }(\cdot )}\leq
\left\Vert w^{-1}\chi _{B_{A}}\right\Vert _{p^{\prime }(\cdot )}$. And by
definition of the $\mathcal{A}_{p(\cdot )}$ class, 
\[
\left\Vert w\chi _{B_{A}}\right\Vert _{p(\cdot )}\left\Vert w^{-1}\chi
_{B_{A}}\right\Vert _{p^{\prime }(\cdot )}\leq [w]_{\mathcal{A}_{p(\cdot)}}\left\vert B_{A}\right\vert.
\]

Since $\left\vert
B_{A}\right\vert =\left\| A\right\| \left\vert B\right\vert $, we
conclude that 
\[
\left\Vert w_{A}\chi _{B}\right\Vert _{p(\cdot )}\left\Vert w_{A}^{-1}\chi
_{B}\right\Vert _{p^{\prime }(\cdot )}\leq [w]_{\mathcal{A}_{p(\cdot)}}\|A\|^n\left\vert B\right\vert.
\]

Therefore $w_{A}\in \mathcal{A}_{p(\cdot )}$ and $[w_A]_{\mathcal{A}_{p(\cdot)}}\leq [w]_{\mathcal{A}_{p(\cdot)}}\|A\|^n$. 
\end{proof}

\subsection{Proof of Theorem \ref{Principal}}

The necessary condition of Theorem \ref{Principal} was proved in the last section.
We focus on the sufficient condition.

\begin{proof}[Proof of Theorem \ref{Principal}]

If $M_{A^{-1}}$ is bounded on $L^{p(\cdot)}_w$ then, by Theorem \ref{Necesidad-w}, $w\in \mathcal{A}_{A,p(\cdot)}$.
Conversely, let
$p(\cdot)\in \mathcal{P}^{\log}(\mathbb{R}^n)$ such that $1<p_-\leq p_+<\infty$ and $p(Ax)=p(x)$ a.e. $x\in \mathbb{R}^n$, $w\in \mathcal{A}_{A,p(\cdot)}$ and $f\in L^{p(\cdot)}(w)$.

Let $b\geq 2^{n+1}$. We consider $\Omega_{A,k}=\{x \in \mathbb{R}^n : M_{A^{-1}}f(x)>b^k\}$
and
$D_{A,k}=\{x \in \mathbb{R}^n : b^k<M_{A^{-1}}f(x)\leq b^{k+1}\}$. 
Observe that $\Omega_{A,k}=A(\{x \in \mathbb{R}^n : M_{A^{-1}}f(x)>b^k\})=A (\Omega_k)$ and $D_{A,k}=A(D_k)$.

By the Calderón-Zygmund decomposition, we have $Q_{j,k}$ such that 
$$b^k <\frac1{|Q_{j,k}|}\int_{Q_{j,k}}|f(x)|\, dx \leq b^{k+1}$$
and 
$$\Omega_A\subset \bigcup_{j} AQ^3_{j,k}$$
where $Q^3_{j,k}$ is the cube with the same center of $Q_{j,k}$ and triple side length.
Let $F_{i,k}=Q_{i,k} \setminus \Omega_{k+1}$, note that the family $\{F_{i,k}\}_{i,k}$ is pairwise disjoint and $|Q_{i,k}|\leq 2|F_{i,k}|$.

Now, if $g\in L^{p'(\cdot)}(w^{-1})$ non-negative function, we have  

\begin{align*}
    \int_{\mathbb{R}^n} M_{A^{-1}}f(x) g(x) \,dx
    &\leq \sum_{k} \int_{D_{A,k}}b^{k+1}g(x) \, dx
    \leq b^{k+1}\sum_{j,k} \int_{AQ_{j,k}^3}g(x) \, dx
\\&\leq b |\det(A)|\sum_{j,k} \left(\frac1{|Q_{j,k}|}\int_{Q_{j,k}} |f|\right) \int_{Q_{j,k}^3}g_A(x) \, dx
\\&\lesssim \sum_{j,k} \left(\frac1{|Q_{j,k}^3|}\int_{Q_{j,k}^3} |f|\right) \left(\frac1{|Q_{j,k}^3|}\int_{Q_{j,k}^3}g_A\right) |F_{j,k}|
\end{align*}

By Corollary \ref{pesos-A-nilp} we have $w\in \mathcal{A}_{p(\cdot)}$. Since the class $\mathcal{A}_{p(\cdot)}$ is left open, Proposition \ref{Ap-OpenLeft}, there exists $s$ such that $\min\left\{\frac1{p_-},\frac1{p'_-}\right\}<s<1$,  $w^{\frac1{s}}\in \mathcal{A}_{sp(\cdot)}$ and $w^{-\frac1{s}}\in \mathcal{A}_{sp'(\cdot)}$. 

We define the exponents $u'(\cdot)=\frac1{s} (sp(\cdot))'$ and $v(\cdot)=\frac1{s}(sp'(\cdot))'$. Since $p(\cdot) \in \mathcal{P}^{\log}(\mathbb{R}^n)$ we have $u(\cdot),v(\cdot) \in \mathcal{P}^{\log}(\mathbb{R}^n).$

Observe that 
\begin{align*}
    \frac{p'(\cdot)}{v'(\cdot)}&=p'(\cdot)(1-s)+1\geq (p')_{-}(1-s)+1>1,\\
    \frac{p(\cdot)}{u(\cdot)}&=p(\cdot)(1-s)+1\geq p_{-}(1-s)+1>1.
\end{align*}

Then, by Theorem \ref{Mq bounded} w have $M_{u(\cdot)}$ is bounded on $L^{p(\cdot)}(\mathbb{R}^n)$ and $M_{v'(\cdot)}$ is bounded on $L^{p'(\cdot)}(\mathbb{R}^n)$.

Then by Hölder 
\begin{align*}
&\int_{\mathbb{R}^n} M_{A^{-1}}f(x) g(x) \, dx  
\lesssim  \sum_{j,k}  \left(\frac1{|Q_{j,k}^3|}\int_{Q_{j,k}^3}|f|\right) \left(\frac1{|Q_{j,k}^3|}\int_{Q_{j,k}^3}  g_A \right) |F_{j,k}|
\\& \lesssim  \sum_{j,k}  |F_{j,k}|  \frac1{|Q_{j,k}^3|}
\|fw \chi_{Q_{j,k}^3}\|_{u(\cdot)} \|w^{-1}\chi_{Q_{j,k}^3}\|_{u'(\cdot)} 
 \frac1{|Q_{j,k}^3|} \|g_Aw_A^{-1}\chi_{Q_{j,k}^3}\|_{v'(\cdot)} \|w_A\chi_{Q_{j,k}^3}\|_{v(\cdot)}
 \\& \lesssim  \sum_{j,k}  |F_{j,k}|  
\frac{\|fw \chi_{Q_{j,k}^3}\|_{u(\cdot)}}{\|\chi_{Q_{j,k}^3}\|_{u(\cdot)}} \frac{\|w^{-1}\chi_{Q_{j,k}^3}\|_{u'(\cdot)}}{\|\chi_{Q_{j,k}^3}\|_{u'(\cdot)}} 
 \frac{\|g_Aw_A^{-1}\chi_{Q_{j,k}^3}\|_{v'(\cdot)}}{\|\chi_{Q_{j,k}^3}\|_{v'(\cdot)}} \frac{\|w_A\chi_{Q_{j,k}^3}\|_{v(\cdot)}}{{\|\chi_{Q_{j,k}^3}\|_{v(\cdot)}}}
\end{align*}

We now claim that
\begin{equation}\label{claim}
    \frac{\|w^{-1}\chi_{Q_{j,k}^3}\|_{u'(\cdot)}}{\|\chi_{Q_{j,k}^3}\|_{u'(\cdot)}} 
 \frac{\|w_A\chi_{Q_{j,k}^3}\|_{v(\cdot)}}{{\|\chi_{Q_{j,k}^3}\|_{v(\cdot)}}}\lesssim  C_w
\end{equation}

Therefore 
\begin{align*}
    \int_{\mathbb{R}^n} M_{A^{-1}}f(x) g(x) \, dx 
    &\lesssim C_w \sum_{j,k}  |F_{j,k}|  \frac{\|fw \chi_{Q_{j,k}^3}\|_{u(\cdot)}}{\|\chi_{Q_{j,k}^3}\|_{u(\cdot)}}  
    \frac{\|g_Aw_A^{-1}\chi_{Q_{j,k}^3}\|_{v'(\cdot)}}{\|\chi_{Q_{j,k}^3}\|_{v'(\cdot)}} 
\\&\lesssim C_w \sum_{j,k}  \int_{F_{j,k}} M_{u(\cdot)}(fw) M_{v'(\cdot)}(g_Aw_A^{-1}) \,dx
\\&\lesssim C_w \int_{\mathbb{R}^n} M_{u(\cdot)}(fw) M_{v'(\cdot)}(g_Aw_A^{-1}) \, dx
\\&\lesssim C_w \|M_{u(\cdot)}(fw)\|_{p(\cdot)} \|M_{v'(\cdot)}(g_Aw_A^{-1})\|_{p'(\cdot)}
\end{align*}
Note that since $p(Ax)=p(x)$ a.e. $x$, by Proposition \ref{UV} we have that 
$$\|g_Aw_A^{-1}\|_{p'(\cdot)}\lesssim \|gw^{-1}\|_{p'(\cdot)}$$
this imply $g_Aw_A^{-1}\in L^{p'(\cdot)}(\mathbb{R}^n)$.

Now, since $M_{u(\cdot)}$ is bounded on $L^{p(\cdot)}(\mathbb{R}^n)$ and $M_{v'(\cdot)}$ is bounded on $L^{p'(\cdot)}(\mathbb{R}^n)$ we obtain 
\begin{align*}
\int_{\mathbb{R}^n} M_{A^{-1}}f(x) g(x) \, dx 
    &\lesssim C_w  \|fw\|_{p(\cdot)}\|gw^{-1}\|_{p'(\cdot)}
\end{align*}

Then, by duality argument we have the desired estimate.

Finally to complete the proof, we must prove that the claim \eqref{claim} holds. Observe that \eqref{claim} is 
\begin{equation*}
   (w^{-1})_{u^{\prime }(\cdot ),Q}(w_{A})_{v(\cdot
),Q} \lesssim C_w.
\end{equation*}
with $Q={Q_{j,k}^3}$.

By Lemma \ref{w_A in Ap} we have $w_{A}\in \mathcal{A}_{p(\cdot )}$. Then $w_{A}^{-1}\in \mathcal{A}_{p^{\prime }(\cdot )}$. And, so by Proposition \ref{Ap-OpenLeft}, $w_{A}^{1/s}\in \mathcal{A}_{sp(\cdot )}$ and $w_{A}^{-1/s}\in \mathcal{A}_{sp^{\prime }(\cdot )}$. Hence we can apply the Lemma \ref{LemaCUDH} and we get
\begin{equation*}
(w^{1/s})_{sp(\cdot ),Q}(w^{-1/s})_{[sp(\cdot )]^{\prime
},Q} \lesssim [w^{1/s}]_{\mathcal{A}_{sp(\cdot)}} 
\end{equation*}

and
\begin{equation*}
(w_{A}^{-1/s})_{sp^{\prime }(\cdot ),Q}(w_{A}^{1/s})_{[sp^{\prime
}(\cdot )]^{\prime },Q} \lesssim [w_A^{-1/s}]_{\mathcal{A}_{sp'(\cdot)}}.
\end{equation*}

So, by Lemma \ref{dilationExpo}, we obtain
\begin{equation*}
(w)_{p(\cdot ),Q}(w^{-1})_{u^{\prime }(\cdot ),Q} \lesssim [w^{1/s}]_{\mathcal{A}_{sp(\cdot)}}^s
\end{equation*}

and
\begin{equation*}
(w_{A}^{-1})_{p^{\prime }(\cdot ),Q}(w_{A})_{v(\cdot
),Q} \lesssim [w_A^{-1/s}]_{\mathcal{A}_{sp'(\cdot)}}^s.
\end{equation*}

Then using Lemma \ref{LemaCUDH} for $w,w_A\in \mathcal{A}_{p(\cdot)}$ and definition of the class $\mathcal{A}_{A,p(\cdot)}$, we finally obtain
\begin{align*}
(w^{-1})_{u^{\prime }(\cdot ),Q}(w_{A})_{v(\cdot ),Q}
&\lesssim (w^{-1})_{u^{\prime }(\cdot ),Q} (w)_{p(\cdot),Q}(w^{-1})_{p'(\cdot),Q} (w_A)_{p(\cdot),Q}(w_A^{-1})_{p(\cdot),Q} (w_{A})_{v(\cdot ),Q}
\\&\leq [w^{1/s}]_{\mathcal{A}_{sp(\cdot)}}^s [w]_{\mathcal{A}_{A,p(\cdot)}}[w_A^{-1/s}]_{\mathcal{A}_{sp'(\cdot)}}^s=:C_w.
\end{align*}


\end{proof}

\section{Properties of the class {$\mathcal{A}_{A,p(\cdot)}$}}

In this section, we prove several properties of the weight class $\mathcal{A}_{A,p(\cdot)}$.
First, recall that in the case without matrix, it is well known that $[w]_{\mathcal{A}_{p}} \geq 1$ and $[w]_{\mathcal{A}_{p(\cdot)}} K_{p(\cdot)}\geq 1$. 
We now study a lower bound of $[w]_{\mathcal{A}_{A,p(\cdot)}}$.

Let $w\in \mathcal{A}_{A,p(\cdot)}$. For a cube $Q$ centered at origin, we have $AQ\cap Q \not = \emptyset$ and 
\begin{align*}
|AQ\cap Q| &\leq K_{p(\cdot)} \|w\chi_{AQ}\|_{p(\cdot)}\|w^{-1} \chi_Q\|_{p'(\cdot)}
\\&\leq K_{p(\cdot)} \max\{1,|\det(A^{-1})|\} \|w_A\chi_{Q}\|_{p(\cdot)}\|w^{-1} \chi_Q\|_{p'(\cdot)}
\\&\leq K_{p(\cdot)} \max\{1,|\det(A^{-1})|\} [w]_{\mathcal{A}_{A,p}}|Q|.
\end{align*}

Then, we obtain 
\[
[w]_{\mathcal{A}_{A,p(\cdot)}} K_{p(\cdot)} \max\{1,|\det(A^{-1})|\} \geq \sup_{Q=Q(0,\ell)} \frac{|AQ\cap Q|}{|Q|}.
\]
where the supremum is taken over all cubes  $Q(0,\ell)$ centered on the origin and side length $\ell$. 

Observe that $B(0,\ell)\subset Q(0,2\ell)$. We define $\tilde{B}=B\left(0, \frac{\ell}{\max\{\|A^{-1}\|,1\}}\right)$.Then $\tilde{B}\subset AQ\cap Q$ and
\[
\frac{|AQ\cap Q|}{|Q|} \geq \frac{|\tilde{B}|}{|Q|}  
= \frac1{2^n\max\left\{\|A^{-1}\|^n,1\right\}}. 
\]
Observe that this lower bound depends only of the matrix $A$ and can be lower than $1$.

We define 
\[ [[w]]_{\mathcal{A}_{A,p(\cdot)}} := 2^n\max\left\{\|A^{-1}\|^{2n},1\right\}  [w]_{\mathcal{A}_{A,p(\cdot)}} K_{p(\cdot)}.
\]

Note that $[[w]]_{\mathcal{A}_{A,p(\cdot)}}\geq 1$ and also $[[w]]_{\mathcal{A}_{A,p(\cdot)}}\geq [w]_{\mathcal{A}_{A,p(\cdot)}}$.

Some examples of a matrix such that $[w]_{\mathcal{A}_{A,p(\cdot)}} \geq 1$ are matrix with eigenvalues greater than $1$. 

The following results are in the spirit of the lemmas present in \cite{CUDH11,CUP25}. 

\begin{lemma}\label{Lemma 3.3}
Let $A$ be a invertible matrix, $p(\cdot)\in \mathcal{P}(\mathbb{R}^n)$ and $w\in \mathcal{A}_{A,p(\cdot)}$. Then, for any cubes $Q$ and all measurable sets $E\subset Q$ we have
\[\frac{|E|}{|Q|} \leq K_{p(\cdot)} [w]_{\mathcal{A}_{A,p(\cdot)}} \frac{\|w\chi_E\|_{p(\cdot)}}{\|w_A\chi_Q\|_{p(\cdot)}}.
\]
\end{lemma}

\begin{proof}
Let $E\subset Q$ and $w\in \mathcal{A}_{A,p(\cdot)}$. Then, using Hölder inequality and the definition of  $\mathcal{A}_{A,p(\cdot)}$ we have
\begin{align*}
|E| &=\int \chi_E w w^{-1} \chi_Q 
\leq K_{p(\cdot)} \|w\chi_E\|_{p(\cdot)} \|w^{-1}\chi_Q\|_{p'(\cdot)}
\\&
\leq K_{p(\cdot)} [w]_{\mathcal{A}_{A,p(\cdot)}} \|w\chi_E\|_{p(\cdot)} \|w_A\chi_Q\|_{p(\cdot)}^{-1}|Q|.
\end{align*}
\end{proof}

The following lemma is an adapted version of the Lemma \ref{Lema-LH_0} to the $AQ$ setting.

\begin{lemma}\label{Lema-LH_0-A}
Let $p(\cdot)\in LH_0(\mathbb{R}^n)$ with $p_+<\infty$ and let $A$ be an invertible matrix $n \times n$. If $p(Ax)=p(x)$, a.e. $x\in\mathbb{R}^n$, then, for every cube $Q\subset\mathbb{R}^n$ and every $x\in Q$,
$$
|AQ|^{p_-(Q)-p_+(Q)}\leq C_A,
$$
where the constant $C_A$ is independent of $Q$. More precisely, we can take
\[
C_A=C_{LH_0}\max\left\{1,|\det A|^{-p_+}\right\}.
\]
\end{lemma}

\begin{proof}
Since $p_-(Q)-p_+(Q)= p_-(Q)-p(x) + p(x)-p_+(Q)$, it is suffices to prove $|AQ|^{p(x)-p_+(Q)}\leq C_A$. The proof of $|AQ|^{p_-(Q)-p(x)}\leq C_A$ is analogous.

Let $Q\subset\mathbb{R}^n$ be a cube and fix $x\in Q$. 
Since $A$ is invertible, 
\begin{align*}
|AQ|^{p(x)-p_+(Q)}
=
|\det A|^{p(x)-p_+(Q)}
|Q|^{p(x)-p_+(Q)}.
\end{align*}

Since $p(\cdot) \in LH_0 (\mathbb{R}^n)$, by Lemma \ref{Lema-LH_0}, there exists a constant $C_{LH_0}>0$, independent of $Q$, such that
$$
|Q|^{p(x)-p_+(Q)}\leq C_{LH_0}.
$$

Therefore,
$$
|AQ|^{p(x)-p_+(AQ)}
\leq
C_{LH_0}|\det A|^{p(x)-p_+(Q)}.
$$

It remains to estimate the factor involving the determinant. Since $p_+<\infty$,
$$
-p_+\leq p(x)-p_+(Q)\leq0.
$$

We distinguish two cases.
If $|\det A|\geq1$, then, since $p(x)-p_+(Q)\leq0$,
$$
|\det A|^{p(x)-p_+(Q)}
\leq1.
$$

If $0<|\det A|<1$, since
$
p(x)-p_+(Q)\geq-p_+,
$
we obtain
$$
|\det A|^{p(x)-p_+(Q)}
\leq|\det A|^{-p_+}.
$$

Thus, in both cases,
$$
|\det A|^{p(x)-p_+(Q)}
\leq
\max\left\{1,|\det A|^{-p_+}\right\}.
$$

Consequently,
$$
|AQ|^{p(x)-p_+(Q)}
\leq
C_{LH_0}\max\left\{1,|\det A|^{-p_+}\right\} =:C_A 
$$
where $C_A$ depends only on $A$, $p(\cdot)$, and the $LH_0$ constant, and is independent of the cube $Q$ and the point $x$.
\end{proof}

\begin{lemma}\label{Lema-w-LH_0-A}
Given $A$ a invertible matrix, $p(\cdot)\in \mathcal{P}^{\log}(\mathbb{R}^n)$ such that  $p_+<\infty$ and $p(Ax)=p(x)$ a.e. $x$ and $w\in \mathcal{A}_{A,p(\cdot)}$ 
then for all cubes $Q$  we have
\[\|w\chi_{AQ}\|_{p(\cdot)}^{p_-(Q)-p_+(Q)} \leq L_1 [[w]]_{\mathcal{A}_{A,p(\cdot)}}^{p_+-p_-}
\]
where 
\[L_1=C(A,n,p(\cdot),C_{\infty}) C_A \max\{1,\|w_A\chi_{Q_0}\|_{p(\cdot)}^{p_- -p_+}\} 
\]
and $Q_0=Q(0,2e)$.
\end{lemma}


\begin{proof}
The proof follows the ideas of \cite{CUDH11,CUP25} with several changes that involve matrix A. 
Consider $w\in \mathcal{A}_{A,p(\cdot)}$.
Fix a cube $Q$ and $Q_0=Q(0,e)$.  
Observe that if $\|w\chi_{AQ}\|_{p(\cdot)}\geq 1$, the desired estimate is immediate. We can assume $\|w\chi_{AQ}\|_{p(\cdot)}< 1$.

Now, we split the proof in several cases according the diameter of $AQ$, side lenght of $Q_0$ and the distance between $AQ$ and $Q_0$.

The first case, suppose diam$(AQ)\leq \ell(Q_0)$ and dist$(AQ,Q_0)\leq \ell(Q_0)$. Then $AQ\subset Q_0^5$.
Then using Lemma \ref{Lemma 3.3} we have
\begin{align*}
|AQ|&\leq K_{p(\cdot)} |5Q_0| [w]_{\mathcal{A}_{A,p(\cdot)}} \|w\chi_{AQ}\|_{p(\cdot)}\|w_A\chi_{5Q_0}\|_{p(\cdot)}^{-1}
\\&\leq |5Q_0| \, [[w]]_{\mathcal{A}_{A,p(\cdot)}} \|w\chi_{AQ}\|_{p(\cdot)}\|w_A\chi_{Q_0}\|_{p(\cdot)}^{-1}.
\end{align*}

Thus,  by Lemma \ref{Lema-LH_0-A} 
\begin{align*}
\|w\chi_{AQ}\|_{p(\cdot)}^{p_-(Q)-p_+(Q)} 
&\leq  \left( \frac{|5Q_0|}{|AQ|} \right)^{p_+(Q)-p_-(Q)} [[w]]_{\mathcal{A}_{A,p(\cdot)}}^{p_+(Q)-p_-(Q)} \|w_A\chi_{Q_0}\|_{p(\cdot)}^{p_-(Q)-p_+(Q)}
\\& \leq   |5Q_0|^{p_+-p_-} C_A [[w]]_{\mathcal{A}_{A,p(\cdot)}}^{p_+-p_-}  \|w_A\chi_{Q_0}\|_{p(\cdot)}^{p_-(Q)-p_+(Q)} .
\end{align*}

For the second case, suppose $\text{diam}(AQ)\leq \ell(Q_0)$ and dist$(AQ,Q_0)>\ell(Q_0)$. Let $\tilde{Q}=Q(0, 3\,\text{dist}(AQ,Q_0))$ then $AQ,Q_0\subset \tilde{Q}$ and $\ell(\tilde{Q})\leq 3\text{dist}(AQ,0)$.

Now, using Lemma \ref{Lemma 3.3} we have
\begin{align*}
|AQ|&\leq K_{p(\cdot)} [w]_{\mathcal{A}_{A,p(\cdot)}} \|w\chi_{AQ}\|_{p(\cdot)} |\tilde{Q}| \|w_A\chi_{\tilde{Q}}\|_{p(\cdot)}^{-1}
\\&\leq  [[w]]_{\mathcal{A}_{A,p(\cdot)}} \|w\chi_{AQ}\|_{p(\cdot)} |\tilde{Q}|  \|w_A\chi_{Q_0}\|_{p(\cdot)}^{-1}.
\end{align*}
Thus, 
\begin{align*}
\|w\chi_{AQ}\|_{p(\cdot)}^{p_-(Q)-p_+(Q)} 
&\leq  [[w]]_{\mathcal{A}_{A,p(\cdot)}}^{p_+(Q)-p_-(Q)} \|w_A\chi_{Q_0}\|_{p(\cdot)}^{p_-(Q)-p_+(Q)} \left(\frac{|\tilde{Q}|}{|AQ|}\right)^{p_+(Q)-p_-(Q)}
\\& \leq  [[w]]_{\mathcal{A}_{A,p(\cdot)}}^{p_+-p_-} \|w_A\chi_{Q_0}\|_{p(\cdot)}^{p_-(Q)-p_+(Q)} |\tilde{Q}|^{p_+(Q)-p_-(Q)} C_A.
\end{align*}
The estimate of $|\tilde{Q}|^{p_+(Q)-p_-(Q)}$ is analogous to the estimate in  \cite{CUP25}. That is, 
\[|\tilde{Q}|^{p_+(Q)-p_-(Q)}\leq 3^{n(p_+-p_-)}e^{2nC_{\infty}}
\]
then 
\[
\|w\chi_{AQ}\|_{p(\cdot)}^{p_-(Q)-p_+(Q)} 
\leq   2^{n(p_+-p_-)}e^{2nC_{\infty}}C_D [[w]]_{\mathcal{A}_{A,p(\cdot)}}^{p_+-p_-} \|w_A\chi_{Q_0}\|_{p(\cdot)}^{p_-(Q)-p_+(Q)}
\]

For the third case, suppose  diam$(AQ)>\ell(Q_0)$ and dist$(AQ,Q_0)\leq \text{diam}(AQ)$, then $AQ,Q_0\subset Q_A=Q(Ac_Q,5\,\text{diam}(AQ))$.

By Lemma \ref{Lemma 3.3} we have 
\begin{align*}
|AQ| 
&\leq K_{p(\cdot)}|Q_A| [w]_{\mathcal{A}_{A,p(\cdot)}} \|w\chi_{AQ}\|_{p(\cdot)}\|w_A\chi_{Q_A}\|_{p(\cdot)}^{-1}
\\&\leq |Q_A| [[w]]_{\mathcal{A}_{A,p(\cdot)}} \|w\chi_{AQ}\|_{p(\cdot)}\|w_A\chi_{Q_0}\|_{p(\cdot)}^{-1}.
\end{align*}
Observe that $1\leq \frac{|Q_A|}{|AQ|}\leq \frac{5^n\|A\|^n\sqrt{n}}{|\det(A)|}=:d_A$. Thus, 
\begin{align*}
\|w\chi_{AQ}\|_{p(\cdot)}^{p_-(Q)-p_+(Q)} 
&\leq d_A^{n(p_+-p_-)}[[w]]_{\mathcal{A}_{A,p(\cdot)}}^{p_+-p_-} \|w_A\chi_{Q_0}\|_{p(\cdot)}^{p_-(Q)-p_+(Q)}.
\end{align*}

For the last case, suppose diam$(AQ)>\ell(Q_0)$ and dist$(AQ,Q_0)> \text{diam}(AQ)$. Consider $\tilde{Q}=Q(0,3\,\text{dist}(AQ,Q_0))$ then $AQ,Q_0\subset \tilde{Q}$.

By Lemma \ref{Lemma 3.3} we have 
\[|AQ| \leq |\tilde{Q}| [[w]]_{\mathcal{A}_{A,p(\cdot)}} \|w\chi_{AQ}\|_{p(\cdot)}\|w_A\chi_{Q_0}\|_{p(\cdot)}^{-1}.
\]
Thus, 
\begin{align*}
\|w\chi_{AQ}\|_{p(\cdot)}^{p_-(Q)-p_+(Q)} 
&\leq [[w]]_{\mathcal{A}_{A,p(\cdot)}}^{p_+(Q)-p_-(Q)} \|w_A\chi_{Q_0}\|_{p(\cdot)}^{p_-(Q)-p_+(Q)}\left(\frac{|\tilde{Q}|}{|AQ|}\right)^{p_+(Q)-p_-(Q)}
\\&\leq [[w]]_{\mathcal{A}_{A,p(\cdot)}}^{p_+-p_-} \|w_A\chi_{Q_0}\|_{p(\cdot)}^{p_-(Q)-p_+(Q)} C_A 2^{n(p_+-p_-)}e^{2nC_{\infty}}.
\end{align*}

Summary of the estimate proved above, for a cube $Q$ fixed, we have 

\[
\|w\chi_{AQ}\|_{p(\cdot)}^{p_-(Q)-p_+(Q)}
\leq C_1(n,p(\cdot),C_{\infty})C_A [[w]]_{\mathcal{A}_{A,p(\cdot)}}^{p_+-p_-} \|w_A\chi_{Q_0}\|_{p(\cdot)}^{p_-(Q)-p_+(Q)}
\]

where $C_1(n,p(\cdot),C_{\infty})$ is the maximum between the constants in all cases.  

Now, we estimate $ \|w_A\chi_{Q_0}\|_{p(\cdot)}^{p_+(Q)-p_-(Q)}$.
If $\|w_A\chi_{Q_0}\|_{p(\cdot)}<1$ then $\|w_A\chi_{Q_0}\|_{p(\cdot)}^{p_-(Q)-p_+(Q)}\leq \|w_A\chi_{Q_0}\|_{p(\cdot)}^{p_--p_+}$. On the other hand, if $\|w_A\chi_{Q_0}\|_{p(\cdot)}\geq 1$ then  $\|w_A\chi_{Q_0}\|_{p(\cdot)}^{p_-(Q)-p_+(Q)}\leq 1$. Then
\[\|w_A\chi_{Q_0}\|_{p(\cdot)}^{p_-(Q)-p_+(Q)}\leq \max\left\{1,\|w_A\chi_{Q_0}\|_{p(\cdot)}^{p_--p_+}\right\}.\]

Hence, 
\[
\|w\chi_Q\|_{p(\cdot)}^{p_-(Q)-p_+(Q)}
\leq C(A,n,p(\cdot),C_{\infty})[[w]]_{\mathcal{A}_{A,p(\cdot)}}^{p_+-p_-} \max\left\{1,\|w_A\chi_{Q_0}\|_{p(\cdot)}^{p_--p_+}\right\}
\]
where 
\[C(A,n,p(\cdot),C_{\infty})=C_1(n,p(\cdot),C_{\infty}) C_A. \]
\end{proof}

Now, we present modular estimates for our class of weights. 

\begin{lemma}\label{Control modular}
Let $A$ be a invertible matrix, $p(\cdot)\in \mathcal{P}^{\log}(\mathbb{R}^n)$ such that $p_+<\infty$ and $p(Ax)=p(x)$ a.e. $x$ and $w\in \mathcal{A}_{A,p(\cdot)}$. 
Then for all cubes $Q$, we have
\[\rho(w_A\chi_Q) \leq  L_2 [[w]]_{\mathcal{A}_{A,p(\cdot)}}^{\max\left\{2p_+-p_-,2C_\infty\frac{p_+}{p_-}+p_{\infty}\right\}}   \rho(w\chi_Q)\]
where
\[L_2=C(A,n,p(\cdot)) \max\left\{\|w_A\chi_{Q_0}\|_{p(\cdot)}^{p_--p_+},\|w_A\chi_{Q_0}\|_{p(\cdot)}^{p_+-p_-}\right\}\]
and $Q_0=Q(0,2e)$. 
\end{lemma}

\begin{proof}
Fix a cube $Q$ and let $Q_0=Q(0,2e)$. 
We consider several cases depending on the size of $\|w_A\chi_Q\|_{p(\cdot)}$ and $\|w\chi_Q\|_{p(\cdot)}$.

The first case is $\|w_A\chi_Q\|_{p(\cdot)},\|w\chi_Q\|_{p(\cdot)} \leq 1$. 
By Lemma \ref{Lemma 3.3} we obtain
\[
1=\frac{|Q|}{|Q|} 
\leq K_{p(\cdot)} [w]_{\mathcal{A}_{A,p(\cdot)}} \frac{\|w\chi_Q\|_{p(\cdot)}}{\|w_A\chi_Q\|_{p(\cdot)}}
\leq [[w]]_{\mathcal{A}_{A,p(\cdot)}} \frac{\|w\chi_Q\|_{p(\cdot)}}{\|w_A\chi_Q\|_{p(\cdot)}}.
\]
By Lemma \ref{norma-modular} we have $\|w\chi_Q\|_{p(\cdot)}\leq  \rho(w\chi_Q)^{\frac1{p_+(Q)}}$ and $\|w_A\chi_Q\|_{p(\cdot)}\geq \rho(w_A\chi_Q)^{\frac1{p_-(Q)}}$. Then
\begin{align*}
1&\leq\left([[w]]_{\mathcal{A}_{A,p(\cdot)}}\frac{\|w\chi_Q\|_{p(\cdot)}}{\|w_A\chi_Q\|_{p(\cdot)}}\right)^{\frac{p_+(Q)}{p_+}}
\\&
\leq [[w]]_{\mathcal{A}_{A,p(\cdot)}}^{\frac{p_+(Q)}{p_+}}\left(\frac{\|w\chi_Q\|_{p(\cdot)}^{p_+(Q)}}{\|w_A\chi_Q\|_{p(\cdot)}^{p_-(Q)}}\right)^{\frac1{p_+}}\|w_A\chi_Q\|_{p(\cdot)}^{\frac{p_-(Q)-p_+(Q)}{p_+}}
\\&\leq [[w]]_{\mathcal{A}_{A,p(\cdot)}} \left(\frac{\rho(w\chi_Q)}{\rho(w_A\chi_Q)}\right)^{\frac1{p_+}}\left(\|w_A\chi_Q\|_{p(\cdot)}^{p_-(Q)-p_+(Q)}\right)^{\frac1{p_+}}.
\end{align*}
Then by Proposition \ref{UV} and Lemma \ref{Lema-w-LH_0-A} we have 
\begin{align*}
\|w_A\chi_Q\|_{p(\cdot)}^{p_-(Q)-p_+(Q)}&=\|(w\chi_{AQ}) \circ A\|_{p(\cdot)}^{p_-(Q)-p_+(Q)}
\lesssim  \|w\chi_{AQ}\|_{p(\cdot)}^{p_-(Q)-p_+(Q)} 
\\&\leq C(A,n,p(\cdot)) [w]_{\mathcal{A}_{A,p(\cdot)}}^{p_+-p_-}.
 \end{align*}

Then, 
\[
\rho(w_A\chi_Q) \leq C(A,n,p(\cdot)) [[w]]_{\mathcal{A}_{A,p(\cdot)}}^{2p_+-p_-} \rho(w\chi_Q).
\]

The second case is $\|w\chi_Q\|_{p(\cdot)} \leq 1 < \|w_A\chi_Q\|_{p(\cdot)}$.
Then,  by Lemma \ref{Lemma 3.3}
\begin{align*}
1=\frac{|Q|}{|Q|} 
&\leq K_{p(\cdot)} [w]_{\mathcal{A}_{A,p(\cdot)}} \frac{\|w\chi_Q\|_{p(\cdot)}}{\|w_A\chi_Q\|_{p(\cdot)}}
\leq  [[w]]_{\mathcal{A}_{A,p(\cdot)}} \left(\frac{\rho(w\chi_Q)}{\rho(w_A\chi_Q)}\right)^{\frac1{p_+}}.
\end{align*}

The third  case is $\|w_A\chi_Q\|_{p(\cdot)} \leq 1 < \|w \chi_Q\|_{p(\cdot)}$.
Then,  by Lemma \ref{Lemma 3.3}
\begin{align*}
1=\frac{|Q|}{|Q|} 
&\leq K_{p(\cdot)} [w]_{\mathcal{A}_{A,p(\cdot)}} \frac{\|w\chi_Q\|_{p(\cdot)}}{\|w_A\chi_Q\|_{p(\cdot)}}
\leq [[w]]_{\mathcal{A}_{A,p(\cdot)}} \left(\frac{\rho(w\chi_Q)}{\rho(w_A\chi_Q)}\right)^{\frac1{p_-}}.
\end{align*}

The fourth and last case is  $\|w_A\chi_Q\|_{p(\cdot)},\|w\chi_Q\|_{p(\cdot)} > 1$.
We proceed as in \cite[Lemma 3.5]{CUP25}, we only mentioned the main changes.

Let $\lambda =\|w_A\chi_Q\|_{p(\cdot)}$ Then by Lemma \ref{Lema-LH_infty} with $d\mu(x) = w_A(x)^{p(x)} \, dx$, for all $t>0$ 
\begin{align*}
\int_Q \lambda^{-p_{\infty}}  w_A(x)^{p(x)} \, dx
&\leq e^{ntC_{\infty}}\int_Q \lambda^{-p(x)}  w_A(x)^{p(x)} \, dx +\int_Q \frac{w_A(x)^{p(x)}}{(e+|x|)^{tnp_-}} \, dx 
\\&=e^{ntC_{\infty}}\rho \left(\frac{w_A\chi_Q}{\lambda}\right) +\int_Q \frac{w_A(x)^{p(x)}}{(e+|x|)^{tnp_-}} \, dx.
\end{align*}

Observe that $\rho \left(\frac{w_A\chi_Q}{\lambda}\right)=1$. 
Now, we need to estimate the second term. 
For each $k\in \mathbb{N}$ define $Q_k=Q(0,2e^{k+1})$. Then 
\[
\int_{\mathbb{R}^n} \frac{w_A(x)^{p(x)}}{e+|x|^{tnp_-}} \, dx \leq e^{-tnp_-}  +\sum_{k=1}^{\infty}  e^{-kntp_-} \|w_A \chi_{Q_k}\|_{p(\cdot)}^{p_+}.
\]

Observe that by Lemma \ref{Lemma 3.3} applied to $Q_0\subset Q_k$ we have
\[
 \|w_A \chi_{Q_k}\|_{p(\cdot)} \leq K_{p(\cdot)}[w]_{\mathcal{A}_{A,p(\cdot)}} \frac{|Q_k|}{|Q_0|} \|w\chi_{Q_0}\|_{p(\cdot)}
\leq [[w]]_{\mathcal{A}_{A,p(\cdot)}} e^{kn} \|w\chi_{Q_0}\|_{p(\cdot)}.
\]
In particular, $1=\|w_A \chi_{Q_0}\|_{p(\cdot)} \leq
[[w]]_{\mathcal{A}_{A,p(\cdot)}}\|w\chi_{Q_0}\|_{p(\cdot)}$.

Taking $t_1=\frac{p_+}{p_-} +\frac1{np_-} \log\left(2\,[[w]]_{\mathcal{A}_{A,p(\cdot)}}^{p_+}\|w\chi_{Q_0}\|_{p(\cdot)}^{p_+}\right)$, we have
\begin{align*}
\int_Q \lambda^{-p_{\infty}}  w_A(x)^{p(x)} \, dx
&\leq e^{nt_1 C_{\infty}}+ \left[ e^{-t_1np_-} + [[w]]_{\mathcal{A}_{A,p(\cdot)}}^{p_+}\|w\chi_{Q_0}\|_{p(\cdot)}^{p_+} \sum_{k=1}^{\infty}  e^{kn(p_+-t_1p_-)}\right]
\\&\leq 2^{1+\frac{C_{\infty}}{p_-}}\left[e^n [[w]]_{\mathcal{A}_{A,p(\cdot)}}\|w\chi_{Q_0}\|_{p(\cdot)}\right]^{C_{\infty}\frac{p_+}{p_-}}.  
\end{align*}
Then, 
\begin{align}\label{eq: control-wA-Q}
\|w_A\chi_Q\|_{p(\cdot)}=\lambda 
\geq \frac{\rho(w_A\chi_Q)^{\frac1{p_{\infty}}}}{C_1(n,p(\cdot))\left[[w]_{\mathcal{A}_{A,p(\cdot)}}\|w\chi_{Q_0}\|_{p(\cdot)}\right]^{\frac{C_{\infty}}{p_{\infty}}\frac{p_+}{p_-}}} .
\end{align}

Similarly, replacing $p_\infty$ and $p(x)$ we have
\[
1=\int_Q \|w\chi_Q\|_{p(\cdot)}^{-p(x)} w(x)^{p(x)} \, dx 
\leq e^{nsC_{\infty}} \int_Q \|w\chi_Q\|_{p(\cdot)}^{-p_{\infty}} w(x)^{p(x)} \, dx 
+\int_Q \frac{w(x)^{p(x)}}{(e+|x|)^{snp_-}}\,dx.
\]
 
We can split $\mathbb{R}^n$ using $AQ_k$ instead of $Q_k$, $\mathbb{R}^n = A Q_0 \cup \bigcup_{k=1 }^{\infty} (AQ_k \setminus AQ_{k-1})$.
Observe that 
\[\rho( w_A \chi_Q) =\rho( (w\chi_{AQ})\circ A)=|\det(A^{-1})|\rho(w\chi_{AQ}).\]
Then for all $s>0$ we have 
\begin{align*}
\int_{\mathbb{R}^n} \frac{w(x)^{p(x)}}{(e+|x|)^{snp_-}} \, dx
&= \int_{AQ_0} \frac{w(x)^{p(x)}}{(e+|x|)^{snp_-}} \, dx +\sum_{k=1}^{\infty} \int_{AQ_k\setminus AQ_{k-1}} \frac{w(x)^{p(x)}}{(e+|x|)^{snp_-}} \, dx
\\&\leq e^{-snp_-}  \rho(w\chi_{AQ_0})+c_A\sum_{k=1}^{\infty}  e^{-knsp_-} \rho(w \chi_{AQ_k})
\\&\leq e^{-snp_-} |\det(A)| \rho(w_A\chi_{Q_0})+ |\det(A)| \sum_{k=1}^{\infty}  e^{-knsp_-} \rho(w_A \chi_{Q_k}).
\end{align*}
Since $\|w_A\chi_{Q_0}\|_{p(\cdot)}=1$ then $\|w_A\chi_{Q_k}\|_{p(\cdot)}\geq 1$ and $\rho(w_A\chi_{Q_k}) \leq \|w_A\chi_{Q_k}\|_{p(\cdot)}^{p_+}$. That implies
\begin{align*}
\int_{\mathbb{R}^n} \frac{w(x)^{p(x)}}{(e+|x|)^{snp_-}} \, dx
\leq e^{-snp_-} |\det(A)| + |\det(A)| \sum_{k=1}^{\infty}  e^{-knsp_-} \|w_A \chi_{Q_k}\|_{p(\cdot)}^{p_+}.
\end{align*}

As in the above estimate, we can take $s_1>\frac{p_+}{p-}$ such that 
\[\int_{\mathbb{R}^n} \frac{w(x)^{p(x)}}{e+|x|^{s_1np_-}} \, dx <\frac1{2}
\quad \text{and} \quad 
e^{s_1nC_{\infty}} = e^{nC_{\infty}\frac{p_+}{p_-}}5^{\frac{C_{\infty}}{p_-}}\left[ |\det(A)|K_{p(\cdot)}[w]_{\mathcal{A}_{A,p(\cdot)}}\right]^{C_{\infty}\frac{p_+}{p_-}}
\]

Then, we have

\[\frac1{2} \leq e^{ns_1C_{\infty}} \int_Q \|w\chi_Q\|_{p(\cdot)}^{-p_{\infty}} w(x)^{p(x)} \, dx  = e^{ns_1C_{\infty}} \frac{\rho(w\chi_Q)}{\|w\chi_Q\|_{p(\cdot)}^{p_{\infty}}} \]

and that imply  
\begin{align}\label{eq: control-w-Q}
\|w\chi_Q\|_{p(\cdot)} \leq  C_2(n,p(\cdot),A) [[w]]_{\mathcal{A}_{A,p(\cdot)}}^{\frac{C_{\infty}}{p_{\infty}}\frac{p_+}{p_-}}
\rho(w\chi_Q)^{\frac1{p_{\infty}}}.
\end{align}

Then, taking account the estimates  \eqref{eq: control-wA-Q} and \eqref{eq: control-w-Q} we have
\begin{align*}
1=\frac{|Q|}{|Q|} 
&\leq K_{p(\cdot)} [w]_{\mathcal{A}_{A,p(\cdot)}} \frac{\|w\chi_Q\|_{p(\cdot)}}{\|w_A\chi_Q\|_{p(\cdot)}}
\leq C(n,p(\cdot),A) [[w]]_{\mathcal{A}_{A,p(\cdot)}}^{2\frac{C_{\infty}}{p_{\infty}}\frac{p_+}{p_-}  +1 }\left(\frac{\rho(w\chi_Q)}{\rho(w_A\chi_Q)}\right)^{\frac1{p_{\infty}}}.
\end{align*}

Summary the four cases, if $\|w_A\chi_{Q_0}\|_{p(\cdot)}=1$, we obtain
\[\rho(w_A\chi_Q) \leq C(A,n,p(\cdot)) [[w]]_{\mathcal{A}_{A,p(\cdot)}}^{\max\left\{2p_+-p_-,2C_\infty\frac{p_+}{p_-}+p_{\infty}\right\}}   \rho(w\chi_Q)\]
where and $C(A,n,p(\cdot))$ is the maximum over four constants independent of $w$.

If $\|w_A\chi_{Q_0}\|_{p(\cdot)}\not=1$, we consider $u=\frac{w}{\|w_A\chi_{Q_0}\|_{p(\cdot)}}$.
If $\|w_A\chi_{Q_0}\|_{p(\cdot)}>1$  then 
\begin{align*}
&\|w_A\chi_{Q_0}\|_{p(\cdot)}^{-p_-}\rho(w_A\chi_Q) 
\leq \rho(u_A\chi_Q) 
\\&\qquad\leq  C(A,n,p(\cdot)) [[w]]_{\mathcal{A}_{A,p(\cdot)}}^{\max\left\{2p_+-p_-,2C_\infty\frac{p_+}{p_-}+p_{\infty}\right\}}   \rho(u\chi_Q)
\\&\qquad\leq  C(A,n,p(\cdot)) [[w]]_{\mathcal{A}_{A,p(\cdot)}}^{\max\left\{2p_+-p_-,2C_\infty\frac{p_+}{p_-}+p_{\infty}\right\}}   \rho(w\chi_Q)\|w_A\chi_{Q_0}\|_{p(\cdot)}^{-p_+}
\end{align*}
Then
\[\rho(w_A\chi_Q) \leq C(A,n,p(\cdot)) [[w]]_{\mathcal{A}_{A,p(\cdot)}}^{\max\left\{2p_+-p_-,2C_\infty\frac{p_+}{p_-}+p_{\infty}\right\}}   \rho(w\chi_Q) \|w_A\chi_{Q_0}\|_{p(\cdot)}^{p_--p_+}. \]
Similarly, if $\|w_A\chi_{Q_0}\|_{p(\cdot)}>1$ then 
\[\rho(w_A\chi_Q) \leq C(A,n,p(\cdot)) [[w]]_{\mathcal{A}_{A,p(\cdot)}}^{\max\left\{2p_+-p_-,2C_\infty\frac{p_+}{p_-}+p_{\infty}\right\}}   \rho(w\chi_Q) \|w_A\chi_{Q_0}\|_{p(\cdot)}^{p_+-p_-}. \]

\end{proof}

\begin{thm}\label{w_A-w}
Let $p\in \mathcal{P}^{\log}$ such that $p_+<\infty$ and $p(Ax)=p(x)$ a.e. $x$ and $w\in \mathcal{A}_{A,p(\cdot)}$. 
Then 
\[w_A(x) \leq L_3 [[w]]_{\mathcal{A}_{A,p(\cdot)}}^{\frac1{p_-}\max\left\{2p_+-p_-,2C_\infty\frac{p_+}{p_-}+p_{\infty}\right\}} w(x)
\]
where
\[
L_3=C(A,n,p(\cdot)) \max\left\{\|w_A\chi_{Q_0}\|_{p(\cdot)}^{1-\frac{p_+}{p_-}},\|w_A\chi_{Q_0}\|_{p(\cdot)}^{\frac{p_+}{p_-}-1}\right\}
\]
and $Q_0=Q(0,2e)$.

\end{thm}

Observe that the exponent $\max\left\{2p_+-p_-,2C_\infty\frac{p_+}{p_-}+p_{\infty}\right\}\frac1{p_-}\geq 1$ and the equality holds if $p(\cdot)$ is constant.

\begin{proof}
Let $x\in \mathbb{R}^n$. Taking a cube $Q$ centered at $x$. 
By Lemma \ref{Control modular} and divided by $|Q|$, we have  
\[ \frac{\rho(w_A\chi_Q)}{|Q|} \leq L_2[[w]]_{\mathcal{A}_{A,p(\cdot)}}^{\max\left\{2p_+-p_-,2C_\infty\frac{p_+}{p_-}+p_{\infty}\right\}} \frac{\rho(w\chi_Q)}{|Q|}\]

Since $w$ is a weight, $w,w_A$ are locally integrable functions, then by the Lebesgue Differentiation Theorem we have
\[ w_A(x)^{p(x)} \leq L_2[[w]]_{\mathcal{A}_{A,p(\cdot)}}^{\max\left\{2p_+-p_-,2C_\infty\frac{p_+}{p_-}+p_{\infty}\right\}} w(x)^{p(x)} \]
Then, 
\begin{align*}
w_A(x)
&\leq L_2^{\frac1{p(x)}}[[w]]_{\mathcal{A}_{A,p(\cdot)}}^{\max\left\{2p_+-p_-,2C_\infty\frac{p_+}{p_-}+p_{\infty}\right\}\frac1{p(x)}} w(x)
\\&\leq L_3[[w]]_{\mathcal{A}_{A,p(\cdot)}}^{\frac1{p_-}\max\left\{2p_+-p_-,2C_\infty\frac{p_+}{p_-}+p_{\infty}\right\}} w(x).
\end{align*}
\end{proof}

Observe that in the case of matrices $A$ such that $A^k=I$ for some $k\in \mathbb{N}$ and $p(Ax)=p(x)$ a.e. $x$, this result and  Corollary \ref{pesos-A-nilp} imply that 
\[\mathcal{A}_{A,p(\cdot)}=\left\{w\in \mathcal{A}_{p(\cdot)} : w(Ax)\lesssim w(x) \;a.e. \, x\in \mathbb{R}^n\right\}.\]

In this case, it can be proved that $M_{A^{-1}}$ is bounded on $L^{p(\cdot)}(w)$ using Theorem \ref{w_A-w} and the boundedness on $L^{p(\cdot)}(w)$ of the Hardy-Littlewood maximal $M$. However, the 
control of the weight constant is worth than the obtained in Theorem \ref{Principal}.

\section{Examples}

In this final section, we establish a lemma that provides a convenient method for constructing weights in the class $\mathcal{A}_{p(\cdot)}$ under suitable assumptions on the exponent. We then illustrate the result by presenting an explicit example of an exponent, a matrix, and a weight satisfying the hypotheses of Theorem \ref{Principal}. The lemma below follows from the weighted boundedness results for the Hardy--Littlewood maximal operator established in \cite{CUDH11}. For completeness, we include the proof.

\begin{prop}\label{lema-Ap(.)}
Let $p(\cdot)\in\mathcal{P}^{\log}$ satisfy
$1<p_-\leq p_+<\infty,$
and let $p_\infty$ denote its limit at infinity.
Suppose that $\alpha,\beta\in\mathbb{R}$ satisfy
\[
-\frac{n}{p(0)}<\alpha<\frac{n}{p'(0)},
\qquad
-\frac{n}{p_\infty}<\beta<\frac{n}{p_\infty'},
\]
where $p_\infty'=p_\infty/(p_\infty-1)$.
Then the weight
\[
w(x)=|x|^\alpha(1+|x|)^{\beta-\alpha},
\qquad x\in\mathbb{R}^n\setminus\{0\},
\]
belongs to $\mathcal{A}_{p(\cdot)}$. 
The value of $w$ at the origin may be chosen arbitrarily
to be positive.
\end{prop}

\begin{proof}

In \cite{CUDH11} the authors shows that if $\sigma(x)=w(x)^{p(x)}$ belong in $\mathcal{A}_{p(\cdot)}^*$ then $w\in \mathcal{A}_{p(\cdot)}$. Recall that $\sigma\in \mathcal{A}_{p(\cdot)}^*$ if
\begin{itemize}
\item  for every $x_0$, fix a cube $Q_{x_0}$ centered in $x_0$ then $\sigma \in \mathcal{A}_{p(x_0)}$ locally on $Q_{x_0}$ 
\item $\sigma \in \mathcal{A}_{p_{\infty}}$ locally on $\mathbb{R}\setminus Q(0,N)$ for some $N>0$. 
\end{itemize}

Let $\sigma(x)=w(x)^{p(x)}.$ We first examine its behavior near the origin.

For $0<|x|< \frac1{2}$,
\[
\frac{\sigma(x)}{|x|^{\alpha p(0)}}
=
|x|^{\alpha(p(x)-p(0))}
(1+|x|)^{(\beta-\alpha)p(x)}.
\]
Since $p(\cdot)\in LH_0$,
\[
|p(x)-p(0)|\,|\log|x||\leq C_0,
\qquad 0<|x|< \frac1{2}.
\]
Consequently,
\[
|x|^{\alpha(p(x)-p(0))}\approx1.
\]
The second factor is also bounded above and below by
positive constants, because $1\leq1+|x|\leq2$ and
$p_+<\infty$. Hence
\[
\sigma(x)\approx |x|^{\alpha p(0)},
\qquad 0<|x|< \frac1{2}.
\]
The assumptions on $\alpha$ are equivalent to
\[
-n<\alpha p(0)<n\bigl(p(0)-1\bigr).
\]
By the classical characterization of power weights,
$|x|^{\alpha p(0)}\in \mathcal{A}_{p(0)}$.
Therefore, $\sigma$ satisfies the $\mathcal{A}_{p(0)}$ condition
on a neighborhood of the origin.

At every point $x_0\neq0$, the function $\sigma$ is
continuous and strictly positive on a sufficiently small
closed cube centered at $x_0$.
It is therefore bounded above and below by positive
constants on that cube. In particular, $\sigma$ satisfies
the local $A_{p(x_0)}$ condition.

Next, for $|x|\geq1$, we write
\[
\frac{\sigma(x)}{|x|^{\beta p_\infty}}
=
|x|^{\beta(p(x)-p_\infty)}
\left(1+\frac{1}{|x|}\right)^{(\beta-\alpha)p(x)}.
\]
The condition $p(\cdot)\in LH_\infty$ gives
\[
|p(x)-p_\infty|\log (|x|)\leq C_{\infty}\frac{\log (|x|)}{\log (e+|x|)} \leq C_{\infty},
\qquad |x|\geq1.
\]
It follows that
\[
\sigma(x)\approx |x|^{\beta p_\infty},
\qquad |x|\geq1.
\]
Moreover, the assumptions on $\beta$ are equivalent to
\[
-n<\beta p_\infty<n(p_\infty-1).
\]
Thus $|x|^{\beta p_\infty}\in \mathcal{A}_{p_\infty}$, and
$\sigma$ satisfies the $\mathcal{A}_{p_\infty}$ condition outside
a sufficiently large cube.

Hence $\sigma \in \mathcal{A}^*_{p(\cdot)}$. Thus, by \cite[Theorem~1.2]{CUDH11}, the Hardy--Littlewood maximal operator is bounded on the variable Lebesgue space $L^{p(\cdot)}(\mathbb{R}^n,\sigma)$ (that is, treating $\sigma$ as a measure). And this implies that the Hardy--Littlewood maximal operator is bounded on $L^{p(\cdot)}_w$. Then by \cite[Theorem~1.3]{CUDH11}, 
$w\in\mathcal{A}_{p(\cdot)}$.
\end{proof}

\

\begin{ejem}
Let $x\in \mathbb{R}^2$ and we take,
\[
A=
\begin{pmatrix}
0 & -1\\
1 & 0
\end{pmatrix},
\qquad
p(x)=2+e^{-|x|^2},
\]
and consider the weight
\[
w(x)=|x|(1+|x|)^{-3/2}, \quad w(0)=1.
\]

The matrix $A$ is orthogonal and satisfies
$A^4=I$. Hence
\[
p(Ax)=p(x),
\qquad
w(Ax)=w(x).
\]

We first verify the regularity assumptions on the exponent.
Notice that
\[
p_-=2,\qquad p_+=3,\qquad p(0)=3,
\qquad p_\infty=2.
\]
Moreover $p(\cdot)$ is globally Lipschitz,
\[
|p(x)-p(y)|
\leq \sqrt{\frac{2}{e}}\,|x-y|.
\]
Thus $p(\cdot)\in LH_0$. Furthermore for all $x\in \mathbb{R}^2$,
\[
|p(x)-p_\infty|=e^{-|x|^2}
\leq \frac{C}{\log(e+|x|)}.
\]
Therefore, $p\in LH_\infty$.

By Proposition \ref{lema-Ap(.)} the radial weight
\[
v_{\alpha,\beta}(x)
=|x|^\alpha(1+|x|)^{\beta-\alpha}
\]
belongs to the weight class
$\mathcal{A}_{p(\cdot)}$ whenever
\[
-\frac{n}{p(0)}<\alpha<\frac{n}{p'(0)},
\qquad
-\frac{n}{p_\infty}<\beta<\frac{n}{p_\infty'}.
\]
In the present example, $n=2$, $\alpha=1$, and
$\beta=-1/2$. These inequalities become
\[
-\frac23<1<\frac43,
\qquad
-1<-\frac12<1.
\]
Hence $w\in\mathcal{A}_{p(\cdot)}$.

Finally, since $w_A=w$, we obtain $[w]_{\mathcal{A}_{A,p(\cdot)}}
= [w]_{\mathcal{A}_{p(\cdot)}}<\infty$.
Therefore, $w\in\mathcal{A}_{A,p(\cdot)}$.
\end{ejem}

\section*{AI Disclosure}
ChatGPT was used for language editing and proofreading. All mathematical results,
arguments, proofs, and conclusions were developed and proved by the authors.


\begin{thebibliography}{99}

\bibitem{CCUF07}
Capone, C., Cruz-Uribe, D., Fiorenza, A.
\newblock The fractional maximal operator and fractional integrals on variable {$L^p$} spaces. 
\newblock{\em Rev. Mat. Iberoam.} 23 (2007), no. 3, 743--770.

\bibitem{CUDH11}
Cruz-Uribe, D., Diening, L.; Hästö, P.,
\newblock The maximal operator on weighted variable Lebesgue spaces. 
\newblock{\em Fract. Calc. Appl. Anal.} 14 (2011), no. 3, 361--374.

\bibitem{CUFlibro}
Cruz-Uribe, D.~V., and Fiorenza, A.,
\newblock Variable Lebesgue spaces: Foundations and harmonic analysis.
\newblock {\em Springer Science {\&} Business Media}, 2013.

\bibitem{CUFN03}
Cruz-Uribe, D., Fiorenza, A., Neugebauer, C. J.,
\newblock The maximal function on variable {$L^p$} spaces. 
\newblock {\em Ann. Acad. Sci. Fenn. Math.} 28 (2003), no. 1, 223--238.

\bibitem{CUFN12}
Cruz-Uribe, D., Fiorenza, A., Neugebauer, C. J.,
\newblock Weighted norm inequalities for the maximal operator on variable Lebesgue spaces. 
\newblock {\em J. Math. Anal. Appl.} 394 (2012), no. 2, 744--760.

\bibitem{CUP25}
Cruz-Uribe, D., Penrod, M., 
\newblock The reverse Hölder inequality for $\mathcal{A}_{p(\cdot)}$ weights with applications to matrix weights. 
\newblock {\em Publ. Mat.} 70 (2026), no. 2, 493--524.

\bibitem{D04}
Diening, L.,
\newblock Maximal function on generalized Lebesgue spaces $L^{p(\cdot)}$. 
\newblock {\em Math. Inequal. Appl.} 7 (2004), no. 2, 245--253.


\bibitem{DHHRlibro}
Diening L., Harjulehto P., H{\"a}st{\"o} P., Ruzicka M.,
\newblock Lebesgue and Sovolev Spaces with variable exponents, 
\newblock {\em Lecture Notes in Mathematics} 2017, Springer-Verlag
Berlin Heidelberg (2011).

\bibitem{IF26}
Ibañez-Firnkorn, G., 
\newblock Weighted boundedness for the maximal operator associated with matrices. 
\newblock {\em arXiv preprint} arXiv:2603.02126 (2026).  

\bibitem{IFR18} 
Ibañez Firnkorn, G. H., Riveros, M. S., 
\newblock Certain fractional type operators with Hörmander conditions. 
\newblock {\em Ann. Acad. Sci. Fenn. Math.} 43 (2018), no. 2, 913-929

\bibitem{IFRV22}
Ibañez-Firnkorn, G. H.; Riveros, M. S., Vidal, R. E., 
\newblock Necessary condition on weights for maximal and integral operators with rough kernels. \newblock {\em Studia Math.} 263 (2022), no. 3, 293--321.

\bibitem{IFV23}
Ibañez-Firnkorn, Gonzalo; Vallejos, Lucas Alejandro. 
\newblock Boundedness of commutators of integral operators of fractional type and $M_{\alpha, L^r\log L}$ maximal operator in variable Lebesgue spaces. 
\newblock {\em J. Geom. Anal.} 33 (2023), no. 11, Paper No. 354, 16 pp.

\bibitem{RU13}
Riveros, M. S., Urciuolo, M., 
\newblock Weighted inequalities for fractional type operators with some homogeneous kernels.
\newblock {\em Acta Math. Sin.} (Engl. Ser.) 29 (2013), no. 3, 449--460.

\bibitem{UV20}
Urciuolo, M., Vallejos, L.,
\newblock $L^{p(\cdot)}$-$L^{q(\cdot)}$ boundedness of some integral operators obtained by extrapolation techniques. 
\newblock {\em Georgian Math. J.} 27 (2020), no. 3, 479--484.

\bibitem{UV20a}
Urciuolo, M. S., Vallejos, L. A., 
\newblock A generalization of the boundedness of certain integral operators in variable Lebesgue spaces. 
\newblock {\em J. Math. Inequal.} 14 (2020), no. 2, 547--557.

\bibitem{UV20b}
Urciuolo, M., Vallejos, V.,
\newblock Integral operators with rough kernels in variable Lebesgue spaces.
\newblock {\em  Acta Math. Hungar.} 162 (2020), no. 1, 105--116.

\end{thebibliography}
\end{document}